\documentclass[11pt]{article}
\usepackage{amsmath,amsfonts}
\usepackage{graphicx}
\usepackage{multicol}
\usepackage{mathtools}
\usepackage[top=1in, bottom=1in, left=1in, right=1in]{geometry}
\usepackage{caption}
\usepackage{subcaption}
\usepackage{color}
\usepackage{fullpage}

\usepackage{latexsym}
\usepackage{amsmath}
\usepackage{amssymb}
\usepackage{amsthm}
\usepackage{epsfig}
\usepackage{xcolor}
\usepackage{graphicx}
\usepackage{bm}
\usepackage{enumitem}

\newcommand{\E}{\mathbb{E}}

\newcommand{\R}{\mathbb{R}} 
\newcommand{\C}{\mathbb{C}}

\newcommand{\eps}{\varepsilon}

\newcommand{\tr}{\text{Tr}}

\newcommand{\sign}{\text{sign}}
\newcommand{\diag}{\text{diag}}

\newcommand{\ones}{\bm{1}}
\newcommand{\bzero}{\boldsymbol 0}

\newcommand{\op}{{\rm op}}

\newtheorem{corollary}{Corollary}
\newtheorem*{corollary*}{Corollary}
\newtheorem{theorem}{Theorem}
\newtheorem*{theorem*}{Theorem}
\newtheorem{lemma}{Lemma}

\newtheorem{definition}{Definition}
\newtheorem*{definition*}{Definition}

\newcommand{\SDP}{{\textup{\rm SDP}}}
\def\proj{\mathrm P}
\def\ddiag{\text{ddiag}}
\def\group{{\mathcal G}}
\def\mZt{{\mathbb Z_2}}

\def\SO{{\rm SO}}

\def\sT{{\mathsf T}}

\def\tu{{\tilde u}}
\def\tsigma{{\tilde \sigma}}
\def\tD{{\tilde D}}
\def\tLambda{{\tilde \Lambda}}
\def\tgrad{{\textup{grad}}}
\def\thess{{\textup{Hess}}}
\def\cO{{\mathcal O}}
\def\rank{{\rm rank}}
\def\func{{f}}
\def\reals{{\mathbb R}}
\def\integers{{\mathbb Z}}
\def\<{\langle}
\def\>{\rangle}
\def\MaxCut{{\rm MaxCut}}
\def\id{{\mathbf I}}
\def\normal{{\sf N}}
\def\GOE{{\rm GOE}}

\def\hu{\hat{u}}
\def\sML{\mbox{\tiny\rm ML}}
\def\Crit{{\rm Cr}}
\def\cM{{\mathcal M}}
\def\RGD{{\textup{Rg}}}
\def\vB{{\mathbf{B}}}
\def\vc{{\mathbf{c}}}

\def\tr{{\textup{Tr}}}
\def\barX{\bar{X}}
\def\oA{\overline{A}}
\def\prob{{\mathbb P}}

\begin{document}

\title{Solving SDPs for synchronization and MaxCut problems \\ via the Grothendieck inequality}

\author{Song~Mei\footnote{Institute for Computational and Mathematical Engineering, Stanford University}
\and
Theodor~Misiakiewicz\footnote{Departement de Physique, Ecole Normale Sup\'erieure}
\and
Andrea~Montanari\footnote{Department of Electrical
Engineering and Department of Statistics, Stanford University}
\and
Roberto~I.~Oliveira\footnote{
Instituto Nacional de Matem\'atica Pura e Aplicada (IMPA)}
}

\maketitle

\begin{abstract}

A number of statistical estimation problems can be addressed by semidefinite programs (SDP). While SDPs are solvable in polynomial time using interior point methods, 
in practice generic SDP solvers do not scale well to high-dimensional problems. In order to cope with this problem, Burer and Monteiro proposed a non-convex rank-constrained
formulation, which has good performance in practice but is still poorly understood theoretically.  

In this paper we study the rank-constrained version of SDPs  arising in MaxCut and in synchronization problems. We establish a
Grothendieck-type inequality that proves that all the local maxima and  dangerous saddle points are within a small multiplicative 
gap from the global maximum. We use this structural information to prove that SDPs can be solved within a known accuracy, by applying the Riemannian trust-region 
method to this non-convex problem, while constraining the rank to be  of order one. 
For the MaxCut problem, our inequality implies that any local maximizer of the rank-constrained SDP provides a $ (1 - 1/(k-1)) \times 0.878$ approximation 
of the MaxCut, when the rank is fixed to  $k$. 

We then apply our results to data matrices generated according to the Gaussian $\integers_2$ synchronization problem,
and the two-groups stochastic block model with large bounded degree.
We prove that the error achieved by local maximizers undergoes a phase transition at the same threshold as for information-theoretically optimal methods.
\end{abstract}

\section{Introduction}

A successful approach to statistical estimation and statistical learning suggests to estimate the object of
interest by solving an optimization problem, for instance motivated by maximum likelihood, or empirical risk minimization. 
In modern applications, the unknown object is often combinatorial, e.g. a sparse vector in high-dimensional regression
or a partition in clustering. In these cases, the resulting optimization problem is computationally intractable and convex relaxations
have been a method of choice for obtaining tractable and yet statistically efficient estimators.

In this paper  we consider the following specific semidefinite program
\begin{equation}\label{SDP}\tag{MC-SDP}
\begin{aligned}
&\text{maximize} & & \langle A, X \rangle\\
&\text{subject to} & & X_{ii} = 1, \quad i \in [n], \\
&&& X \succeq 0\, ,
\end{aligned}
\end{equation}
as well as some of its generalizations. This SDP famously arises as a convex relaxation of the MaxCut problem\footnote{In the MaxCut problem,
we are given a graph $G=(V,E)$ and want to partition the vertices in two sets as to maximize the number of edges across the partition.}, whereby the matrix $A$ 
is the opposite of the adjacency matrix of the graph to be cut. In a seminal paper,  Goemans and Williamson \cite{goemans1995improved}
proved that this SDP provides a $0.878$ approximation of the combinatorial problem. 
Under the unique games conjecture, this approximation factor is optimal for polynomial time algorithms \cite{khot2007optimal}. 

More recently, SDPs of this form (see below for generalizations) have been studied in the context of group synchronization and community detection problems. An incomplete
list of references includes \cite{singer2011angular, singer2011three,
  bandeira2014multireference, guedon2016community, montanari2016semidefinite, javanmard2016phase, hajek2016achieving, abbe2016exact}. 
In community detection, we try to partition the vertices of a graph into tightly connected communities under a statistical model for the edges. 
Synchronization aims at estimating $n$ elements $g_1,\ldots, g_n$ in a group $\group$, from the pairwise noisy measurement of the group differences
 $g_i^{-1}g_j$.  Examples include $\mZt$ synchronization in which $\group = \mZt=(\{+1,-1\},\;\cdot\;)$ (the group with elements $\{+1,-1\}$ and usual multiplication), 
angular synchronization in which  
$\group = U(1)$ (the multiplicative group of complex numbers of modulo one), and $\SO(d)$ synchronization in which we need to estimate $n$ rotations $R_1, \ldots, R_n \in \SO(d)$ from the special orthogonal group. In this paper, we will focus on $\mZt$ synchronization and $\SO(d)$ synchronization. 

Although SDPs can be solved  to arbitrary precision in polynomial time \cite{nesterov2013introductory}, generic solvers do not
scale well to large instances. In order to address the scalability problem,
\cite{burer2003nonlinear} proposed to reduce the problem dimensions by imposing the rank constraint $\rank(X)\le k$. This constraint can be 
solved by setting $X = \sigma \sigma^\sT$ where $\sigma \in \R^{n \times k}$. In the case of (\ref{SDP}),
we obtain the following non-convex problem, with decision variable $\sigma$:
\begin{equation}\label{NcvxSDP}\tag{$k$-Ncvx-MC-SDP}
\begin{aligned}
&\text{maximize} & & \langle \sigma, A \sigma \rangle\\
&\text{subject to} & & \sigma = [\sigma_1, \ldots, \sigma_n]^\sT \in \mathbb R^{n \times k},\\
&& & \Vert \sigma_{i} \Vert_2 = 1, \quad i \in [n]. \\
\end{aligned}
\end{equation}
Provided that $k\geq \sqrt{2n}$, the solution of (\ref{SDP}) corresponds to the global maximum of (\ref{NcvxSDP}) \cite{barvinok1995problems,pataki1998rank,burer2003nonlinear}. 
Recently,  \cite{boumal2016non} proved that, as long as $k \geq \sqrt{2n}$, for almost all matrices $A$,
the problem  (\ref{NcvxSDP}) has a unique local maximum which is also the global maximum. This paper proposed to use the Riemannian trust-region 
method to solve the non-convex SDP problem, and provided computational complexity guarantees on the resulting algorithm.

While the theory of \cite{boumal2016non} suggests the choice $k = \cO(\sqrt n)$, it has been observed empirically that setting
$k=\cO(1)$ yields excellent solutions and scales well to large scale applications \cite{javanmard2016phase}.
 In order to explain this phenomenon, \cite{bandeira2016low} considered the $\mZt$ synchronization problem
with  $k=2$, and established theoretical guarantees for the local maxima, provided the noise level is small enough.
A different point of view was taken in a recent  unpublished technical note \cite{montanari2016grothendieck}, which proposed a Grothendieck-type
 inequality for the local maxima of (\ref{NcvxSDP}). In this paper we continue and develop the preliminary work in \cite{montanari2016grothendieck},
to obtain explicit computational guarantees for the non-convex approach with rank constraint $k=\cO(1)$.

As mentioned above, we extend our analysis beyond the MaxCut type problem (\ref{NcvxSDP}) to treat an optimization
problem  motivated by  $\SO(d)$ synchronization.
$\SO(d)$ synchronization (with $d=3$) has applications to computer vision \cite{arie2012global} and cryo-electron microscopy (cryo-EM) \cite{singer2011three}. 
A natural SDP relaxation of the maximum likelihood estimator is given by the problem
\begin{equation}\label{OC-SDP}\tag{OC-SDP}
\begin{aligned}
&\text{maximize} & & \langle A, X \rangle\\
&\text{subject to} & & X_{ii} = \id_d, \quad i \in [m], \\
& & & X \succeq 0,\\
\end{aligned}
\end{equation}
with decision variable $X$.
Here $A,X\in\reals^{n\times n}$, $n=md$ are matrices with $d\times d$ blocks denoted by $(A_{ij})_{1\le i,j\le m}$, $(X_{ij})_{1\le i,j\le m}$. 
This semidefinite program is also known as Orthogonal-Cut SDP. In the context of $\SO(d)$ synchronization, $A_{ij} \in \R^{d \times d}$ is a noisy measurement 
of the pairwise group differences $R_i^{-1} R_j$ where $R_i \in \SO (d)$. 

By imposing the rank constraint $\rank(X)\le k$, we obtain a non-convex analogue of  (\ref{OC-SDP}), namely:
\begin{equation}\label{Ncvx-OC-SDP}\tag{$k$-Ncvx-OC-SDP}
\begin{aligned}
&\text{maximize} & & \langle \sigma, A\sigma \rangle\\
&\text{subject to} & & \sigma = [\sigma_1, \ldots, \sigma_m]^{\sT} \in \R^{n\times k},\\
& & & \sigma_i^\sT \sigma_i = \id_d, \quad i \in [m].
\end{aligned}
\end{equation}
Here the decision variables are matrices $\sigma_i\in\reals^{k\times d}$.

According to the result in \cite{burer2003nonlinear}, as long as $k \geq (d+1) \sqrt{m}$, the global maximum of the problem (\ref{Ncvx-OC-SDP}) coincides
 with the maximum of the problem (\ref{OC-SDP}). As proved in \cite{boumal2016non}, with the same value of $k$ for almost all
matrices $A$, the non-convex problem  has no local maximum other than the global maximum. \cite{boumal2015Riemannian} 
proposed to choose the rank $k$ adaptively: as $k$ is not large enough, increase $k$ to find a better solution.
 However, none of these works considers $k=\cO(1)$, which is the focus of the present paper (under the assumption that  $d$ is of order one as well).

\subsection{Our contributions}

A main result of our paper is a Grothendieck-type inequality that generalizes and strengthens the preliminary  technical result of \cite{montanari2016grothendieck}.
Namely, we prove that for  any $\eps$-approximate concave point $\sigma$ of the rank-$k$ non-convex SDP (\ref{NcvxSDP}), we have
\begin{align}
\SDP (A) \geq \func( \sigma) \geq \SDP(A) - \frac{1}{k-1} (\SDP(A) + \SDP(-A)) - \frac{n}{2}\eps\, ,
\end{align}
where $\SDP(A)$ denotes the maximum value of the problem (\ref{SDP}) and $\func(\sigma)$ is the objective function in  (\ref{NcvxSDP}). 
An $\eps$-approximate concave point is a point at which the eigenvalues of the Hessian of $\func(\,\cdot\,)$ are upper bounded by $\eps$ (see below for formal definitions).

Surprisingly, this result connects a second order local property, namely the highest local curvature of the cost function, to its global position. 
In particular, all the local maxima (corresponding to $\eps=0$) are within a $1/k$-gap of the SDP value. 
Namely, for any local maximizer $\sigma^*$, we have
\begin{align}
\func( \sigma^*) \geq \SDP(A) - \frac{1}{k-1} (\SDP(A) + \SDP(-A)) \, .
\end{align}
All the points outside this gap, with an $n\eps/2$-margin have a direction of positive curvature of at least size $\eps$.

Figure \ref{fig:landscape} illustrates the landscape of the rank-$k$ non-convex MaxCut SDP problem (\ref{NcvxSDP}). We show that this structure implies 
global convergence rates for approximately solving (\ref{NcvxSDP}). We study the Riemannian trust-region method in Theorem \ref{thm:trustregion}. 
In particular, we show that this algorithm with any initialization returns a $0.878\times (1-O(1/k))$ approximation of the MaxCut of a random $d$-regular graph in 
$\mathcal{O}(n k^2)$ iterations, cf. Theorem \ref{thm:MaxCut}.

\begin{figure}[ht]
\centering
\includegraphics[width=0.8\textwidth]{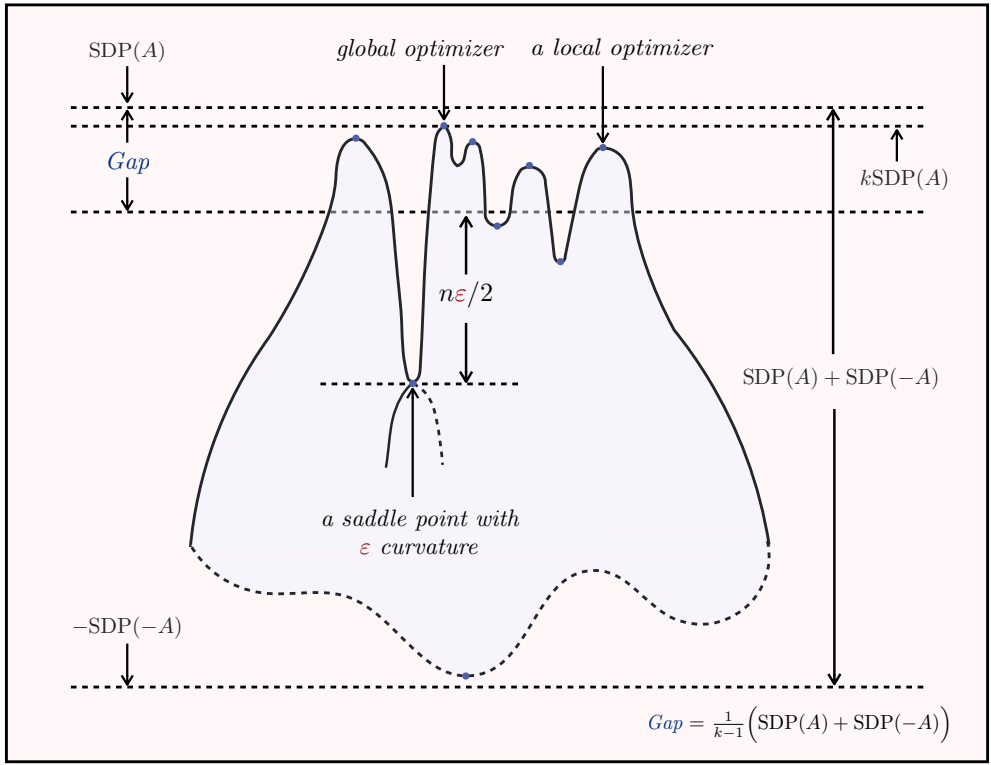}
\caption{The landscape of rank-$k$ non-convex SDP}\label{fig:landscape}
\end{figure}

In the case of $\mZt$ synchronization, we show that for any signal-to-noise  ratio $\lambda >1$, all the local maxima of the rank-$k$ non-convex SDP correlate non-trivially with the 
ground truth  when $k\ge k^*(\lambda) =\cO(1)$ (Theorem \ref{thm:Z2sync1}). Furthermore, Theorem \ref{thm:correlationlimit} provides a lower bound on the 
correlation between local maxima and the ground truth that converges to one when $\lambda$ goes to infinity. 
These results improve over the earlier ones of \cite{bandeira2016low,boumal2016non},
by establishing the tight phase transition location, and the correct qualitative behavior. We extend these
results to the two-groups symmetric Stochastic Block Model. 

For $\SO(d)$ synchronization, we consider the problem (\ref{Ncvx-OC-SDP}) and generalize our main Grothendieck-type inequality
to this case, cf. Theorem \ref{thm:SOd_grothendieck}. Namely, for any $\eps$-approximate concave point $\sigma$ of the rank-$k$ non-convex Orthogonal-Cut SDP (\ref{Ncvx-OC-SDP}), we have
\begin{align}
\func( \sigma) \geq \SDP_o(A) - \frac{1}{k_d-1} (\SDP_o(A) + \SDP_o(-A)) - \frac{n}{2}\eps,
\end{align}
where $k_d = 2k/(d+1)$, $\SDP_o(A)$ denotes the maximum value of the problem (\ref{OC-SDP}) and $\func(\sigma)$ is the objective function in  (\ref{Ncvx-OC-SDP}). We expect that the statistical analysis of local maxima, as well as the analysis of optimization algorithms, should extend to this case as well, but we leave this to future work.

\subsection{Notations}\label{sec:notations}

Given a  matrix $A = (A_{ij}) \in \R^{m \times n}$, we write $\Vert A \Vert_1 = \max_{1 \leq j \leq n} \sum_{i=1}^m \vert A_{ij} \vert$ for
its operator $\ell_1$-norm, $\Vert A \Vert_{\op}$ for its operator $\ell_2$-norm (largest singular value), and $\Vert A \Vert_F = (\sum_{i=1}^m \sum_{j= 1}^n A_{ij}^2)^{1/2}$
 for its Frobenius norm. For two matrices $A, B \in \R^{m \times n}$, we write $\langle A, B \rangle = \tr (A^\sT B)$ for the inner product associated to the Frobenius norm $\langle A , A \rangle = \Vert A \Vert_F^2$. In particular for two vectors $u,v \in \R^n$, $\langle u , v \rangle$ corresponds to the inner product of the vectors $u$ and $v$ associated to the Euclidean norm on $\R^n$. 
We denote by $\ddiag(B)$ the matrix obtained from $B$ by setting to zero all the entries outside the diagonal.

Given a real symmetric matrix $A \in \mathbb R^{n \times n}$, we write $\SDP(A)$ for value of the SDP problem (\ref{SDP}). That is,  
\begin{align}
\SDP(A) = \max\{ \langle A, X\rangle : X \succeq 0, X_{ii} = 1, i \in [n] \}.
\end{align}
Optimization is performed over the convex set of positive-semidefinite matrices with diagonal entries equal to one, also known as the \emph{elliptope}. We write $\RGD (A) = \SDP (A) + \SDP (-A)$ for the length of the range of the SDP with data $A$ (noticing that for every matrix $X$ in the elliptope, 
we have $\SDP(A) \geq \langle A, X \rangle \geq - \SDP(-A)$). 

For the rank-$k$ non-convex SDP problem (\ref{NcvxSDP}), we define the manifold $\cM_k$ as
\begin{align}
\cM_k = \{\sigma \in \mathbb R^{n\times k}: \sigma = (\sigma_1, \sigma_2, \ldots, \sigma_n )^\sT, \Vert \sigma_i \Vert_2 = 1\} \cong \underbrace{\mathbb{S}^{k-1}\times \mathbb{S}^{k-1} \times \ldots \times \mathbb{S}^{k-1}}_{n \text{ times}}.\label{eq:Mkdef}
\end{align}
where $\mathbb{S}^{k-1} \equiv \lbrace x \in \R^k : \Vert x \Vert_2 = 1  \rbrace$ is the unit sphere in $\R^k$. Given a real symmetric matrix $A \in \mathbb R^{n \times n}$, for $\sigma \in \cM_k$, we write $\func( \sigma) = \langle \sigma, A \sigma \rangle$ the objective function of the rank-$k$ non-convex SDP (\ref{NcvxSDP}).  
Our optimization algorithm makes use of the Riemannian gradient and the Hessian of the function $f$. We anticipate their formulas here,
deferring to Section \ref{sec:geometry} for further details. Defining $\Lambda = \ddiag \big(A\sigma\sigma^\sT \big)$, the gradient is given by:
\begin{align}
\tgrad \func( \sigma) = 2 \big(A - \Lambda\big) \, .
\end{align}
The Hessian is uniquely defined by the following holding for all $u,v$ in the tangent space $T_{\sigma} \mathcal{M}_k$:
\begin{align}
\langle v, \thess \func( \sigma)[u] \rangle = 2 \langle v, (A - \Lambda) u \rangle,
\end{align}

\section{Main results}

First we define the notion of approximate concave point of a function $f$ on a manifold $\cM$.
\begin{definition}[Approximate concave point]
Let $f$ be a twice differentiable function on a Riemannian manifold $\cM$. We say $\sigma \in \cM$ is an $\eps$-approximate concave point of $f$ on $\cM$, if $\sigma$ satisfies
\[
\langle u, \thess f(\sigma)[u] \rangle \leq \eps \langle u, u\rangle, \quad \forall u\in T_{\sigma} \cM,
\]
where $\thess f(\sigma)$ denotes the Riemannian (intrinsic) Hessian of $f$ at point $\sigma$,  $T_{\sigma} \cM$ is the tangent space,
and  $\<\,\cdot\,,\cdot\,\>$ is the scalar product on $T_{\sigma} \cM$.
\end{definition}
Note that an approximate concave point may not be a stationary point, or may not even be an approximate stationary point. Both local maximizers
and saddles with largest eigenvalue of the Hessian close to zero are approximate concave points.

The classical Grothendieck inequality relates the global maximum of a non-convex optimization problem to the maximum of 
its SDP relaxation \cite{grothendieck1996resume,khot2012grothendieck}. Our main tool is instead an inequality that applies to all approximate concave ponts in the non-convex problem.
\begin{theorem}
\label{thm:apprxconcavepoint}
For any $\eps$-approximate concave point $\sigma \in \cM_k$ of the rank-$k$ non-convex problem (\ref{NcvxSDP}), we have
\begin{align}
\func( \sigma) \geq \SDP(A) - \frac{1}{k-1} (\SDP(A) + \SDP(-A)) - \frac{n}{2}\eps.
\end{align}
\end{theorem}

\subsection{Fast Riemannian trust-region algorithm}

We can use the structural information in  Theorem \ref{thm:apprxconcavepoint}, to develop an algorithm that approximately solves the problem (\ref{NcvxSDP}),
and hence the MaxCut SDP (\ref{SDP}). The algorithm we propose is a variant of the Riemannian trust-region algorithm. 

The Riemannian trust-region algorithm (RTR) \cite{absil2007trust} is a generalization of the trust-region algorithm to manifolds. To maximize
the objective function $\func$ on the manifold $\cM$, RTR  proceeds as
follows:  at each step, we find a direction $\xi \in T_\sigma \cM$ that maximizes the quadratic approximation of $\func$ over a ball of small radius $\eta_\sigma$
\begin{align}
\tag{RTR-update} \label{RTR-update}
\xi_*\equiv \arg\max\Big\{ \func( \sigma) + \langle \tgrad  \func( \sigma), \xi \rangle + \langle \xi, \thess \func( \sigma) [\xi]\rangle, \quad 
\xi \in T_{\sigma} \cM\, ,\;\;  \Vert \xi \Vert \leq \eta_{\sigma}\big\},
\end{align}
where $\tgrad \func(\sigma)$ is the manifold gradient of $\func$, and the radius $\eta_\sigma$ is chosen to ensure that the higher order terms remain small. The next iterate $\sigma^{\rm new} = \proj_{\cM} (\sigma + \xi^*)$ is obtained by projecting $\sigma + \xi^*$ back onto the manifold.

Solving the trust-region problem (\ref{RTR-update}) exactly is computationally expensive. In order to obtain a faster algorithm, 
we adopt two variants in the RTR algorithm. First, if the gradient of $\func$ at the current estimate $\sigma^t$ is
sufficiently large, we only use gradient information to determine the new direction: we call this a \emph{gradient-step}; if the gradient is small (i.e. we are at an approximately stationary point), we try to maximize uniquely the Hessian contribution:
we call this an \emph{eigen-step}. Second, in an eigen-step, we only approximately maximize the Hessian contribution. Let us emphasize that these two variants are commonly used and
we do not claim they are novel. 

For the non-convex MaxCut SDP problem (\ref{NcvxSDP}), we describe the algorithm concretely as follows. In each step, first we find a direction $u^t$ using the direction-finding routine outlined 
below\footnote{Throughout the paper, points $\sigma\in\cM_k$ and vectors $u\in T_\sigma \cM_k$ are represented by matrices $\sigma,u\in \reals^{n\times k}$ and
hence the norm on $T_\sigma \cM_k$ is identified with the Frobenius norm $\|u\|_F$.}. 

\begin{center}
	\begin{tabular}{ll}
	\hline
	\vspace{-0.35cm}\\
	\multicolumn{2}{l}{ {\sc Direction-Finding Algorithm}}\\
	\hline
	\vspace{-0.35cm}\\
	\multicolumn{2}{l}{ {\bf Input :}  Current position $\sigma^t$; parameter $\mu_G$;}\\
	\multicolumn{2}{l}{ {\bf Output :}  Searching direction $u^t$ with $\Vert u^t \Vert_F = 1$;}\\
        1: & Compute $\Vert \tgrad \func( \sigma^t) \Vert_F$;\\
	2: & If $\Vert \tgrad \func( \sigma^t) \Vert_F > \mu_G$\\
        3: &\phantom{aa} Return $u^t =  \tgrad \func( \sigma^t)/\Vert \tgrad \func( \sigma^t)\Vert_F$;\\
	4: &Else\\
	5: &\phantom{aa} Use power method to construct a direction $u^t \in T_\sigma \cM_k$ such that \\
	    &\phantom{aa} $\Vert u^t \Vert_F = 1$, $\langle u^t, \thess \func( \sigma^t) [u^t] \rangle \geq \lambda_{\max}(\thess \func(\sigma^t))/2$, \\
	    & \phantom{aa} and $\langle u^t, \tgrad \func( \sigma^t) \rangle \geq 0$; Return $u^t$. \\
        6: & End\\
	\vspace{-0.35cm}\\
	\hline
	\end{tabular}
\end{center}

Given this direction $u^t$, we update our current estimate by
 $\sigma^{t+1}= \proj_{\cM_k} (\sigma^t + \eta^t u^t)$ with $\eta^t$ an appropriately chosen step size. We consider two specific implementations for the parameter $\mu_G$ and the choice of step size:

\begin{enumerate}
\item[${\boldsymbol (a)}$] Take $\mu_G = \infty$, which means that only eigen-steps are used. In this implementation, we take the step size $\eta_H^t =  \langle u^t, \thess \func(\sigma^t) [u^t] \rangle/(100 \Vert A \Vert_1)$. 

\item[${\boldsymbol (b)}$] Take $\mu_G = \Vert A \Vert_2$. When $\Vert \tgrad \func(\sigma^t) \Vert_F > \mu_G$, we choose the step size $\eta_G^t = \mu_G / (20 \Vert A \Vert_1)$.
When $\Vert \tgrad \func(\sigma^t) \Vert_F \leq \mu_G$, we choose the step size $\eta_H^t  =   \min \{ \sqrt{\lambda_H^t/(216 \Vert A \Vert_1)};\lambda_H^t/(12 \Vert A \Vert_2) \}$, where $\lambda_H^t = \langle u^t, \thess \func(\sigma^t) [u^t] \rangle$. 
\end{enumerate}

In each eigen-step, we need to compute a direction $u \in T_\sigma \cM_k$ such that $\Vert u \Vert_F = 1$ and $\langle u, \thess \func( \sigma) [u] \rangle \geq \lambda_{\max}(\thess \func(\sigma))/2$. 
This can be done using the following power method. (Note that the condition $\<u^t,\tgrad f(\sigma^t)\>\ge 0$ can always be ensured eventually by replacing $u^t$ by $-u^t$.)
\begin{center}
	\begin{tabular}{ll}
	\hline
	\vspace{-0.35cm}\\
	\multicolumn{2}{l}{ {\sc Power Method}}\\
	\hline
	\vspace{-0.35cm}\\
	\multicolumn{2}{l}{ {\bf Input :}  $\sigma$, $\thess \func(\sigma)$; parameters $N_H$, $\mu_H$;}\\
	\multicolumn{2}{l}{ {\bf Output :}  $u \in T_\sigma \cM_k$, such that $\Vert u \Vert_F = 1$ and $\langle u, \thess \func( \sigma) [u] \rangle \geq \lambda_{\max}(\thess \func(\sigma))/2$;}\\
        1: & Sample a $u^0$ uniformly randomly on $T_\sigma \cM_k$ with $\Vert u^0 \Vert = 1$;\\
	2: & For $i = 1,\ldots, N_H$\\
        3: &\phantom{aa} $u^{i} \leftarrow \thess \func( \sigma) [u^{i-1}] + \mu_H \cdot u^{i-1}$;\\
        4: &\phantom{aa} $u^{i} \leftarrow u^{i}/ \Vert u^{i} \Vert_F$;\\
	5: &End\\
        6: & Return $u^{N}$.\\
	\vspace{-0.35cm}\\
	\hline
	\end{tabular}
\end{center}

The shifting parameter $\mu_H$ can be chosen as $4 \Vert A \Vert_1$ which is an upper bound of $\Vert \thess \func(\sigma)
\Vert_{\op}$. We take the parameter $N_H = C \cdot \Vert A \Vert_1 \log n/\lambda_{\max}(\thess \func(\sigma))$ with a large absolute
constant $C$. In practice, when choosing the parameter $N_H$, we do not know $\lambda_{\max}(\thess \func(\sigma))$ for each $\sigma$, but
we can replace it by a lower bound, or estimate it using some heuristics. It is  a classical result that --with high probability--
the power method with this number of iterations finds a solution $u^t$ with the required curvature \cite{kuczynski1992estimating}. %(see \cite{golub2000eigenvalue} for an historical account,  and  \cite[Theorem 1]{lecturenotes} for a recent exposition). 

\begin{theorem}\label{thm:trustregion}
There exists a universal constant $c$ such that, for any matrix $A$ and $\eps >0$, the Fast Riemannian Trust-Region method with step size as described above for each iteration and initialized with any $\sigma_0 \in \mathcal{M}_k$ returns a point $\sigma^* \in \cM_k$ with
\begin{align}
\func( \sigma^*) \geq \SDP(A) - \frac{1}{k-1} (\SDP(A) + \SDP(-A)) - \frac{n}{2} \eps,\label{eq:trustregion}
\end{align}
within the following number of steps with each implementation
\begin{enumerate}
\item[$(a)$] Taking $\mu_G=\infty$ (i.e. only eigen-steps are used), then it is sufficient to run  $T _H\leq c \cdot n \Vert A \Vert_1^2/\eps^2$ steps. 
\item[$(b)$] Taking $\mu_G = \Vert A \Vert_2$, then it is sufficient  to run $T = T_H + T_G$ steps in which there are $T_H \leq c \cdot  n \max \left( \Vert A \Vert_2^2/\eps^2, \Vert A \Vert_1/\eps \right)$ eigen-steps and $T_G \leq c \cdot  \RGD (A) \Vert A \Vert_1/\Vert A \Vert_2^2$ gradient-steps. 
\end{enumerate}
\end{theorem}
The gap $\RGD(A) / (k-1) = (\SDP(A)+\SDP(-A))/(k-1)$ in Eq.~(\ref{eq:trustregion}), is due to the fact that Theorem \ref{thm:apprxconcavepoint} does not rule out the presence of local maxima 
within an  interval $\RGD(A) / (k-1)$ from the global maximum.
It is therefore natural to set $\eps = 2\RGD(A)/(n(k-1))$, to obtain the following corollary.
\begin{corollary}\label{cor:RTRapprox}
There exists a universal constant $c$ such that for any matrix $A$, the Fast Riemannian Trust-Region method with step size as described above
for each iteration and initialized with any $\sigma_0 \in \mathcal{M}_k$ returns a point $\sigma^* \in \cM_k$ with
\begin{align}\label{eq:approx}
\func( \sigma^*) \geq \SDP(A) - \frac{2}{k-1} (\SDP(A) + \SDP(-A)) 
\end{align}
within the following number of steps with each implementation
\begin{enumerate}
\item[$(a)$] Taking $\mu_G=\infty$ , then it is sufficient to run  $T_H \leq c \cdot n k^2 \left( n \Vert A \Vert_1/\RGD(A) \right)^2$ eigen-steps. 
\item[$(b)$] Taking $\mu_G = \Vert A \Vert_2$, then it is sufficient  to run $T = T_H + T_G$ steps in which there are $T_H \leq c \cdot  n \max \left( n^2 k^2 \Vert A \Vert_2^2/\RGD(A)^2, nk \Vert A \Vert_1/\RGD(A) \right)$ eigen-steps and $T_G \leq c \cdot  \RGD (A) \Vert A \Vert_1/\Vert A \Vert_2^2$ gradient-steps. 
\end{enumerate}
\end{corollary}

In order to develop some intuition on these complexity bounds, let us consider two specific examples.

Consider the problem of finding the minimum bisection of a random $d$-regular graph $G$, with adjacency matrix $A_G$. A natural SDP relaxation is given by the SDP 
(\ref{SDP}) with $A = A_G-\E A_G = A_G-(d/n)\ones\ones^{\sT}$ the centered adjacency matrix. For this choice of $A$, we have $\Vert A \Vert_1 \le 2d$, $\Vert A \Vert_2 = 2 \sqrt{d-1}(1+o_n(1))$ \cite{friedman2003proof}, 
$\SDP(A) = 2n \sqrt{d-1} + o(n)$ and $\SDP(-A) = 2 n \sqrt{d-1} + o(n)$ \cite{montanari2016semidefinite} (with high probability). Using implementation $(a)$ (only eigen-steps), the bound on the 
number of iterations  in Corollary \ref{cor:RTRapprox} scales as $T_H = \mathcal{O}(n d k^2)$. In implementation $(b)$, we choose $\mu_G = \Theta(\sqrt{d})$, and the number of gradient-steps and eigen-steps scale respectively as 
$T_G = \mathcal{O}(n \sqrt d)$  and $T_H=\mathcal{O}(n k \max(k, \sqrt d))$. In terms of floating point operations, in each gradient-step, the computation of the gradient costs $\mathcal{O} (ndk)$ operations; in each eigen-step, each iteration of the power method costs $\mathcal{O} (ndk)$ operations and the number of iterations in each power method scales as $\mathcal{O} (k \sqrt d \log n)$. 
Implementation $(b)$ presents a better scaling. The total number of floating point operations to find a $(1 - O(1/k))$ approximate solution of the minimum bisection SDP of a random $d$-regular graph 
is (with high probability) upper bounded by $\mathcal{O}( n^2 k^3 d^{3/2} \max(k,\sqrt d)\log n)$.

As a second example, consider the MaxCut problem for a $d$-regular graph $G$, with adjacency matrix $A_G$. This can be addressed by considering the SDP (\ref{SDP}) with $A = -A_G$,
and the corresponding non-convex version (\ref{NcvxSDP}). As shown in the next section, finding a $2\RGD(A)/(n(k-1))$-approximate concave point of (\ref{NcvxSDP}) yields an $(1 - O(1/k)) \times 0.878$-approximation of the MaxCut of $G$. 
For this choice of $A$, we have $\Vert A \Vert_1 =d$, $\Vert A \Vert_2 = d$, and $\RGD(A)= \Theta(nd)$. Therefore, in implementation $(a)$ where all the steps are eigen-step, the number of iterations given by Corollary \ref{cor:RTRapprox} scales as $T_H = \mathcal{O}(n k^2)$. In implementation $(b)$, we choose $\mu_G = \Theta(d)$, and the number of gradient-steps and eigen-steps scale respectively as $T_G = \mathcal{O}(n)$  and $T_H=\mathcal{O}(nk^2)$. In terms of floating point operations, the computational costs of one gradient-step and one eigen-step power iteration are the same (which are $\mathcal{O}(ndk)$) as in the example of minimum bisection SDP. The number of iterations in  the power method scales 
as $\mathcal{O} (k \log n)$. Therefore, the two approaches are equivalent. The total number of floating point operations to find a $(1 - O(1/k)) \times 0.878$ approximate solution of the MaxCut of a 
$d$-regular graph is upper bounded by $\mathcal{O}( n^2 d k^4 \log n)$.

Let us emphasize that the complexity bound in Theorem \ref{thm:trustregion}  is not superior to the ones available for some alternative approaches.
There is a vast literatures that studies fast  SDP solvers \cite{arora2005fast, arora2007combinatorial, steurer2010fast,
 garber2011approximating}. In particular, \cite{arora2007combinatorial, steurer2010fast}  give nearly linear-time algorithms to approximate (\ref{SDP}).
These algorithms are different from the one studied here, and rely on the multiplicative weight update method \cite{arora2012multiplicative}. 
Using sketching techniques, their  complexity can be further reduced \cite{garber2011approximating}.
However, in practice,  the Burer-Monteiro approach studied here  is extremely simple and scales well to large instances \cite{burer2003nonlinear,javanmard2016phase}.
Empirically, it appears to have better complexity than what is guaranteed by our theorem. 
 It would be interesting to compare the multiplicative weight update method and the non-convex approach both theoretically and experimentally. 

\subsection{Application to MaxCut}

Let  $A_G \in \R^{n \times n}$ denote the weighted adjacency matrix of a non-negative weighted graph $G$. The  MaxCut of $G$ is given by the following integer program
\begin{align}\label{eq:MaxCut}
\MaxCut(G) = \max_{x_i \in \{ -1, +1\}} \frac{1}{4}\sum_{i,j=1}^n A_{G,ij} (1-x_i x_j). 
\end{align}
We consider the following semidefinite programming relaxation
\begin{align}\label{eq:SDPCut}
\text{SDPCut}(G) = \max_{X \succeq 0, X_{ii}=1} \frac{1}{4} \sum_{i,j=1}^n A_{G,ij}(1 - X_{ij}).
\end{align}
Denote by $X^*$ the solution of this SDP. Goemans and Williamson \cite{goemans1995improved} proposed a 
celebrated rounding scheme using this $X^*$, which is guaranteed to find an $\alpha_*$-approximate solution to the MaxCut problem (\ref{eq:MaxCut}),
where $\alpha_* \equiv \min_{\theta\in[0,\pi]}2\theta/(\pi(1-\cos\theta))$, $\alpha_*>0.87856$.

The corresponding rank-$k$ non-convex formulation is given by
\begin{equation}\label{eq:NcvxSDP}
\max_{\sigma} \left\{\frac{1}{4} \sum_{i,j=1}^n A_{G,ij} (1-\langle \sigma_i, \sigma_j\rangle)\, :\;\;
 \sigma_i \in \mathbb S^{k-1}, \forall i\in [n]\right\}\, .
\end{equation}
Applying Theorem \ref{thm:apprxconcavepoint}, we obtain the following result. 
\begin{theorem}\label{thm:MaxCut}
For any $k \ge 3$, if $\sigma^*$ is a local maximizer of the rank-$k$ non-convex SDP problem (\ref{eq:NcvxSDP}), then using $\sigma^*$ we can find an $\alpha_* \times (1-1/(k-1)) \ge 0.878 \times (1-1/(k-1))$-approximate solution of the MaxCut problem (\ref{eq:MaxCut}). If $\sigma^*$ is a $2\RGD(A_G)/(n(k-1))$-approximate concave point, then using $\sigma^*$ we can find an $\alpha_*\times (1-2/(k-1)) \ge 0.878 \times (1-2/(k-1))$-approximate solution of the MaxCut problem. 
\end{theorem}
The proof is deferred to Section \ref{sec:proof_thm3}.

\subsection{$\mathbb Z_2$ synchronization}

Recall the definition of the Gaussian Orthogonal Ensemble. We write $W \sim \GOE(n)$ if $W \in \R^{n \times n}$ is symmetric
with $(W_{ij})_{i\le j}$ independent, with distribution $W_{ii} \sim \normal (0,2/n)$ and $W_{ij} \sim \normal (0,1/n)$ for $i < j$. 

In the $\mathbb Z_2$ synchronization problem, we are required  to estimate the vector $u \in  \lbrace \pm 1 \rbrace^{n}$ from noisy pairwise measurements
\begin{equation}\label{spiked}
A(\lambda) = \frac{\lambda}{n} u u^\sT + W_n\, ,
\end{equation}
where $W_n \sim \GOE(n)$, and  $\lambda$ is a signal-to-noise ratio. 
The random matrix model (\ref{spiked}) is also known as the `spiked model' \cite{johnstone2001distribution} %\sm{Added reference to Johnstone, instead of Johnstone Lu}
or `deformed Wigner matrix' 
and has attracted significant   attention across statistics and probability theory \cite{baik2005phase}.

The Maximum Likelihood Estimator for recovering the labels $u \in \lbrace \pm 1 \rbrace^n$ is given by
\[
\hu^{\sML}(A) = \arg\max_{x \in \lbrace \pm 1 \rbrace^n}\<x,Ax\>\, .
\] 
A natural SDP relaxation of this optimization problem is given --once more-- by  (\ref{SDP}). 

It is known that $\integers_2$ synchronization undergoes a phase transition at $\lambda_c=1$.
For $\lambda\le 1$, no statistical estimator $\hu(A)$ achieves scalar product $|\<\hu(A),u\>|$ bounded away from $0$ 
as $n\to\infty$. For $\lambda>1$, there exists an estimator with $|\<\hu(A),u\>|$ bounded away from $0$ (`better than random guessing')
\cite{korada2009exact, deshpande2015asymptotic}. Further, for $\lambda<1$ it is not possible\footnote{To the best of our knowledge, a formal proof of this statement has not been published. 
However, a proof can be obtained by the techniques of \cite{mossel2015reconstruction}.}
to distinguish whether $A$ is drawn from the spiked model or $A\sim \GOE(n)$ with probability of
error converging to $0$ as $n\to\infty$. This is instead possible for $\lambda\ge 1$.

It was proved in \cite{montanari2016semidefinite} that the SDP relaxation (\ref{SDP}) --with a suitable rounding scheme--
achieves the information-theoretic threshold $\lambda_c=1$ for this problem.
 In this paper, we prove a similar result for the non-convex problem (\ref{NcvxSDP}). Namely, we show that for any signal-to-noise ratio $\lambda >1$ there exists a sufficiently large $k$ such that every local maximizer has a non trivial correlation to the ground truth.
Below we denote by $\Crit_{n,k}(A)$ the set of  local maximizers of problem (\ref{NcvxSDP}).
\begin{theorem}\label{thm:Z2sync1}
For any $\lambda > 1$, there exists a function $k_*(\lambda) > 0$, such that for any $k > k_*(\lambda)$, with high probability, any local maximizer $\sigma$ of the rank-$k$ non-convex SDP (\ref{NcvxSDP}) problem has non-vanishing correlation with the ground truth parameter. Explicitly, there exists $\eps = \eps(\lambda) > 0$ such that
\begin{align}
\lim_{n\to \infty}\prob\left(\inf_{\sigma \in \mathbb \Crit_{n,k}(A)}
  \frac{1}{n} \Vert \sigma^\sT u \Vert_2 \geq \eps\right) =1 \, .
\end{align}
\end{theorem}
The proof of this theorem is deferred to Section \ref{sec:proof_thm4}.

Note that this guarantee is weaker than the one of \cite{montanari2016semidefinite}, which also presents an explicit rounding scheme to
obtain an estimator $\hu\in\{+1,-1\}^n$. However, we expect that the techniques of  \cite{montanari2016semidefinite} should be generalizable to
the present setting. A simple rounding scheme takes the sign of principal left singular vector of $\sigma$. 
We will use this estimator in our numerical experiments in Section \ref{numsim}. 

This theorem can be compared with the one of \cite{bandeira2016low} which uses $k=2$ but requires $\lambda > 8$. As a side result which improves over \cite{bandeira2016low} for $k = 2$, we obtain the following lower bound on the correlation for any $k \geq 2$. 
\begin{theorem}\label{thm:correlationlimit}
For any $k \geq 2$, the following holds almost surely
\begin{align}
\liminf_{n\rightarrow \infty} \inf_{\sigma \in \Crit_{n,k}} \frac{1}{n^2} \Vert \sigma^\sT u \Vert_2^2& \geq 1 - \min\left(\frac{16}{\lambda}, \frac{1}{k}  + \frac{4}{\lambda}\right)\, .
\end{align}
\end{theorem}
The proof is deferred to Section \ref{sec:proof_thm7}. Our lower bound converges to $1$ at large $\lambda$, which is the qualitatively correct behavior.

\subsection{Stochastic block model}

The planted partition problem (two-groups symmetric stochastic block model), is another well-studied statistical estimation problem that can be reduced to (\ref{SDP}) \cite{montanari2016semidefinite}. 
We write $G\sim \mathcal{G} (n,p,q)$ if $G$ is a graph over $n $ vertices generated as follows (for simplicity of notation, we assume $n$ even).  Let $u \in \lbrace \pm 1 \rbrace^n$ be a vector of labels that is uniformly random with $u^\sT \ones = 0$. Conditional on this partition,  edges are drawn independently with
\[
\mathbb{P}((i,j) \in E\vert u) = \left\{
    \begin{array}{ll}
        p, & \mbox{if } u_i = u_j, \\
        q, & \mbox{if } u_i \neq u_j.
    \end{array}
\right.
\]
We consider the case when $p = a/n$ and $q = b/n$ with $a,b=\cO(1)$, and $a > b$, and denote by  $d = (a+b)/2$ the average degree. A 
phase transition occurs as the following signal-to-noise parameter increases
\[
\lambda (a,b) \equiv \frac{a-b}{\sqrt{2(a+b)}} \, .
\]
For $\lambda>1$ there exists an efficient estimator that correlates with the true labels with high probability \cite{massoulie2014community,mossel2013proof}, 
whereas no  estimator exists below this threshold, regardless of its computational complexity \cite{mossel2015reconstruction}. 

The Maximum Likelihood Estimator of the vertex labels is given by
\begin{equation}\label{MLE}\tag{SBM-MLE}
\hu^{\sML}(G) = \arg\max\big\{\<x,A_G x\>:\; \;\; x\in\{+1,-1\}^n, \; \< x,\ones\> =0\big\}\, ,
\end{equation}
where $A_G$ is the adjacency matrix of the graph $G$. This optimization problem can again be attacked using the relaxation (\ref{SDP}), 
where $A = \oA_G\equiv (A_G - d/n \cdot \ones \ones^\sT )/\sqrt{d}$ is the scaled and centered adjacency matrix. 

In order to emphasize the relationship between this problem and $\mZt$ synchronization, we rewrite $\oA_G= (\lambda /n) u u^\sT + E$ where $E^\sT = E$ has zero mean and $(E_{ij})_{i <j}$ are independent with distribution
\[
E_{ij} = \left\{
    \begin{array}{ll}
        \frac{1}{\sqrt{d}} ( 1- p_{ij}), & \text{with probability } p_{ij}, \\
        -\frac{p_{ij}}{\sqrt{d}}, & \text{with probability } 1-p_{ij},
    \end{array}
\right.
\]
where $p_{ij} = a/n$ for $u_i = u_j$ and $p_{ij} = b/n$ for $u_i \neq u_j$. In analogy with Theorem \ref{thm:Z2sync1}, we have the following results on the rank-constrained 
approach to the two-groups stochastic block model. 
\begin{theorem}\label{thm:SBM1}
Consider the rank-$k$ non-convex SDP  (\ref{NcvxSDP}) with $A = \oA_G$ the centered, scaled adjacency matrix of graph  $G \sim \mathcal{G} (n, a/n, b/n)$.
For any $\lambda  = \lambda (a,b) > 1$, there exists an average degree $d_* (\lambda)$ and a  rank $k_*(\lambda)$, such that for any $d\geq d_* (\lambda)$ and $k \geq k_*(\lambda)$, 
with high probability, any local maximizer $\sigma$ has non-vanishing correlation with the true labels. Explicitly, there exists an $\eps = \eps(\lambda) > 0$ such that
\begin{align}
\lim_{n\to\infty}\prob\Big(
\inf_{\sigma \in \Crit_{n,k}} \frac{1}{n} \Vert \sigma^\sT u \Vert_2^2 \geq \eps\Big) = 1\, .
\end{align}
\end{theorem}
The proof of this theorem can be found in Section \ref{sec:proof_sbm}. As mentioned above, efficient algorithms that estimate the hidden partition better than random guessing 
for $\lambda>1$ and any $d>1$ have been developed, among others, in \cite{massoulie2014community,mossel2013proof}. However, we expect
 the optimization approach (\ref{NcvxSDP}) to share some of the robustness properties of semidefinite programming \cite{moitra2016robust}, while scaling well to large instances.

\subsection{$\SO(d)$ synchronization}

In $\SO(d)$ synchronization we would like to estimate $m$ matrices $R_1,\dots,R_m$ in the special orthogonal group
\[
\SO(d) = \{ R \in \R^{d \times d}: R^\sT R = \id_d,  \det(R) = 1 \} \, , 
\]
from noisy measurements of the pairwise group differences $A_{ij} = R_i^{-1} R_j + W_{ij}$ for each pairs $(i,j) \in [m] \times [m]$. Here $A_{ij} \in \reals^{d\times d}$ is a measurement, and $W_{ij}\in\reals^{d\times d}$ is
noise.

The Maximum Likelihood Estimator for recovering the group elements $R_i \in \SO(d)$ solves the problem of the form
\[
\max_{\sigma_1\dots \sigma_m\in \SO(d)} \sum_{i,j=1}^m \<\sigma_i,A_{ij}\sigma_j\>\, , 
\]
which can be relaxed to the Orthogonal-Cut SDP (\ref{OC-SDP}). The non-convex rank-constrained approach fixes  $k> d$, and solves the problem (\ref{Ncvx-OC-SDP}). 
This is a smooth optimization problem with objective function $\func(\sigma) = \langle \sigma, A \sigma \rangle$ over the manifold $\cM_{o,d,k}= O(d,k)^m$,
where $O(d,k) =\{\sigma\in\reals^{k\times d}:\; \sigma^{\sT}\sigma = \id_d\}$ is the set of $k\times d$ orthogonal matrices.
We also denote the maximum value of the SDP  (\ref{OC-SDP}) by
\begin{align}
\SDP_o(A) = \{\langle A, X \rangle : \;\;X \succeq 0, \; X_{ii} = \id_d, \; i \in [m] \}.
\end{align}

In analogy with the MaxCut SDP, we  obtain the  following Grothendieck-type inequality.
\begin{theorem}\label{thm:SOd_grothendieck}
For an $\eps$-approximate concave point $\sigma \in \cM_{o,d,k}$ of the rank-$k$ non-convex Orthogonal-Cut SDP problem (\ref{Ncvx-OC-SDP}), we have
\begin{align}
\func( \sigma) \geq \SDP_o(A) - \frac{1}{k_d-1} (\SDP_o(A) + \SDP_o(-A)) - \frac{n}{2}\eps
\end{align}
where $k_d = 2k/(d+1)$.
\end{theorem}
The proof of this theorem is a generalization of the proof of Theorem \ref{thm:apprxconcavepoint}, and is deferred to Section \ref{sec:proof_sod}. 
%%
%\begin{theorem}\label{thm:SOd_grothendieck}
%For any second-order concave point $\sigma$ of the rank-$k$ non-convex Orthogonal-Cut SDP problem (\ref{Ncvx-OC-SDP}), we have
%%
%\begin{align}
%\func( \sigma) \geq \SDP_o(A) - \frac{1}{k_d-1} (\SDP_o(A) + \SDP_o(-A)) 
%\end{align}
%where $k_d = 2k/(d+1)$.
%\end{theorem}
%
%Here a second-order concave point is a point at which $\thess f(\sigma)\preceq 0$. In particular, local maxima are second-order concave, but we make no assumption on $\tgrad f(\sigma)$. 

%\sm{Theodor suggests that we remark here that for any $\eps$-approximate concave point, the above equation holds with just adding $-\frac{1}{2}n\eps$. }

\section{Proof of Theorem \ref{thm:apprxconcavepoint}}

In this section we present the proof of Theorem \ref{thm:apprxconcavepoint}, while deferring other proofs to Section \ref{sec:proofs}. Notice that the present proof is simpler and provides a tighter bound with respect to 
the one of \cite{montanari2016grothendieck}. Before passing to the actual proof, we make a few remarks about the geometry of optimization on $\cM_k$.

\subsection{Geometry of the manifold $\cM_k$}\label{sec:geometry}

The set $\mathcal{M}_k$ as defined in (\ref{eq:Mkdef}) is a smooth submanifold of $ \mathbb{R}^{n\times k}$. We endow $\mathcal{M}_k$ with the Riemannian geometry induced by the Euclidean space $\R^{n \times k}$. At any point $\sigma \in \mathcal{M}_k$, the tangent space is obtained by taking the differential of the equality constraints
\[
T_{\sigma} \mathcal{M}_k = \big\lbrace u \in \mathbb{R}^{n\times k}:u = \left( u_1,u_2, \ldots, u_n \right)^\sT ,  \left\langle u_i , \sigma_i \right\rangle = 0, i \in [n] \big\rbrace \, .
\]
In words, $T_{\sigma}\cM_k$ is the set of matrices $u\in \R^{n\times k}$ such that each row $u_i$ of $u$ is orthogonal to the corresponding row $\sigma_i$
of $\sigma$. Equivalently, $T_{\sigma} \mathcal{M}_k$ is the direct product of the tangent spaces of the $n$ unit spheres $\mathbb S^{k-1}\subseteq\R^k$ at $\sigma_1$,\dots, $\sigma_n$. 
Let $\proj^\perp$ be the orthogonal projection operator from $\R^{n \times k}$ onto $T_{\sigma} \mathcal{M}_k$. We have
\[
\begin{aligned}
\proj^\perp (u) &= \big( \proj^\perp_1 (u_1),\ldots, \proj^\perp_n (u_n)  \big)^\sT \\
& =  \big( u_1 - \<\sigma_1,u_1\>\sigma_1, \ldots ,u_n - \<\sigma_n,u_n\>\sigma_n \big)^\sT \\
& =  u - \text{ddiag} \big( u \sigma^\sT \big) \sigma,
\end{aligned}
\]
where we denoted by $\ddiag : \R^{n \times n} \rightarrow \R^{n \times n}$ the operator on the matrix space that sets all off-diagonal entries to zero. 

In problem (\ref{NcvxSDP}), we consider the cost function $\func( \sigma ) = \langle \sigma , A \sigma \rangle$ on the submanifold $\cM_k$. At $\sigma \in \cM_k$, we denote $\nabla \func( \sigma)$ and $\tgrad \func( \sigma)$ 
respectively the Euclidean gradient in $\R^{n \times k }$ and the Riemannian gradient of $\func$. 
The former is $\nabla f(\sigma) = 2A\sigma$, and the latter  is the projection of the first onto the tangent space:
\[
\tgrad \func( \sigma) = \proj^\perp (\nabla \func( \sigma)) = 2 \Big(A - \ddiag \big(A\sigma\sigma^\sT \big) \Big) \sigma\, .
\]
We will write $\Lambda = \Lambda(\sigma) = \text{ddiag} \left(A \sigma \sigma^\sT \right)$ and often drop the dependence on $\sigma$ for simplicity. 
At $\sigma \in \cM_k$, let $\nabla^2 \func( \sigma)$ and $\thess \func( \sigma)$ be respectively the Euclidean and the Riemannian Hessian of $\func$. The Riemannian Hessian is a symmetric operator on the tangent space and is given by projecting the directional derivative of the gradient vector field (we use $D$ to denote the directional derivative):
\[
\forall u \in T_{\sigma} \mathcal{M}_k, \quad \thess \func( \sigma) [u] = \proj^\perp \left( D \tgrad \func( \sigma) [u] \right) = \proj^\perp \left[ 2(A - \Lambda)u - 2 \ddiag \left( A \sigma u^\sT + A u \sigma^\sT \right) \sigma \right].
\]
In particular, we will use the following identity
\begin{align}\label{eq:hess_expression}
\forall u,v \in T_{\sigma}\cM_k, \quad
\langle v, \thess \func( \sigma)[u] \rangle = 2 \langle v, (A - \Lambda ) u \rangle,
\end{align}
where we used that the projection operator $\proj^\perp$ is self-adjoint and $\langle v_i, \sigma_i \rangle = 0$ by definition of the tangent space.
We observe that the Riemannian Hessian has a similar interpretation as in Euclidean geometry, namely it provides a second order approximation of the function $f$ in a neighborhood of $\sigma$. 

\subsection{Proof of Theorem  \ref{thm:apprxconcavepoint}}

Let $\sigma$ be an $\eps$-approximate concave point of $\func( \sigma)$ on $\cM_k$.  Using the definition and Equation (\ref{eq:hess_expression}), we have
(for $\Lambda  = \text{ddiag}(A \sigma \sigma^\sT )$)
\begin{align} \label{eq:apprx}
\forall u \in T_{\sigma}\cM_k, \quad \langle u, (\Lambda  - A) u \rangle \geq -\frac{1}{2}\eps \langle u, u \rangle\, .
\end{align}

Let $V = [v_1, \ldots, v_n]^\sT \in \mathbb R^{n\times n}$ be such that $X = V V^\sT$ is an optimal solution of (\ref{SDP}) problem. Let $G \in \mathbb R^{k\times n}$ be a random matrix with independent entries $G_{ij} \sim \normal(0, 1/k)$,
and denote by $\proj_i^\perp = \id_k - \sigma_i  \sigma_i^\sT \in \mathbb R^{k \times k}$ the projection onto the subspace orthogonal to $\sigma_i $ in $\R^k$. 

We use $G$ to obtain a random projection  $W = [\proj_1^\perp G v_1,  \ldots, \proj_n^\perp G v_n]^\sT \in T_{\sigma}\cM_k$. From (\ref{eq:apprx}), we have
\[
\E \langle W, (\Lambda  -A) W\rangle \geq -\frac{1}{2}\eps\,  \E \langle W, W\rangle, 
\]
where the expectation is taken over the random matrix $G$. 

The left hand side of the last equation gives
\begin{align*}
&\E \langle W, (\Lambda  - A) W \rangle\\
=& \E \sum_{i, j =1}^n (\Lambda  - A)_{ij} \langle \proj_i^\perp G v_i, \proj_j^\perp G v_j\rangle\\
=& \E\sum_{i, j =1}^n (\Lambda  - A)_{ij} \left\langle \proj_i^\perp G \sum_{s=1}^n v_{is} e_s, \proj_j^\perp G \sum_{t=1}^n v_{jt} e_t \right\rangle\\
=& \sum_{i, j =1}^n (\Lambda  - A)_{ij} \sum_{s,t=1}^n v_{is} v_{jt} \E [\langle \proj_i^\perp G e_s, \proj_j^\perp G e_t\rangle]\\
=& \sum_{i, j =1}^n (\Lambda  - A)_{ij} \sum_{s,t=1}^n v_{is} v_{jt} \delta_{st} \frac{1}{k} \tr(\proj_i^\perp \proj_j^\perp)\\
=& \sum_{i, j =1}^n (\Lambda  - A)_{ij} \langle v_i, v_j \rangle \frac{1}{k} \tr \left( \id_k - \sigma_i \sigma_i^\sT - \sigma_j \sigma_j^\sT + \sigma_i \sigma_i^\sT\sigma_j \sigma_j^\sT \right)\\
=& \sum_{i, j =1}^n (\Lambda  - A)_{ij} \langle v_i, v_j \rangle \left(1 - \frac{2}{k} + \frac{1}{k} \langle \sigma_i, \sigma_j \rangle ^2 \right)\\
=& \left(1-\frac{1}{k}\right)\tr(\Lambda ) - \left(1-\frac{2}{k}\right) \SDP(A) - \frac{1}{k}\sum_{i,j=1}^n A_{ij} \langle v_i, v_j \rangle \, (\langle \sigma_i, \sigma_j \rangle)^2, 
\end{align*}
whereas the right hand side verifies
\[
\E \langle W, W\rangle = \E\sum_{i=1}^n \langle \proj_i^\perp G v_i, \proj_i^\perp G v_i \rangle = \sum_{i=1}^n \left(1- \frac{2}{k} + \frac{1}{k} \Vert \sigma_i \Vert^2_2 \right) = \left(1-\frac{1}{k}\right) n. 
\]

Note that $\tr(\Lambda) = \func( \sigma)$. Crucially, if we let $\barX_{ij} = \langle v_i, v_j \rangle (\langle \sigma_i, \sigma_j \rangle)^2$, we have $\barX_{ii} = 1$ and $\barX \succeq 0$.
 Thus we have $\SDP(-A) \geq \langle -A, \barX \rangle $. Therefore, we have
\[
\left(1-\frac{1}{k}\right) \func( \sigma) - \left(1-\frac{2}{k}\right) \SDP(A) + \frac{1}{k} \SDP(-A) \geq -\frac{1}{2} \eps n \left(1-\frac{1}{k} \right).
\]
Rearranging the terms gives the conclusion.

\section{Numerical illustration}
\label{numsim}

\begin{figure}[ht]
\centering
\begin{minipage}{0.45\linewidth}
\centering
\includegraphics[width=0.95\textwidth]{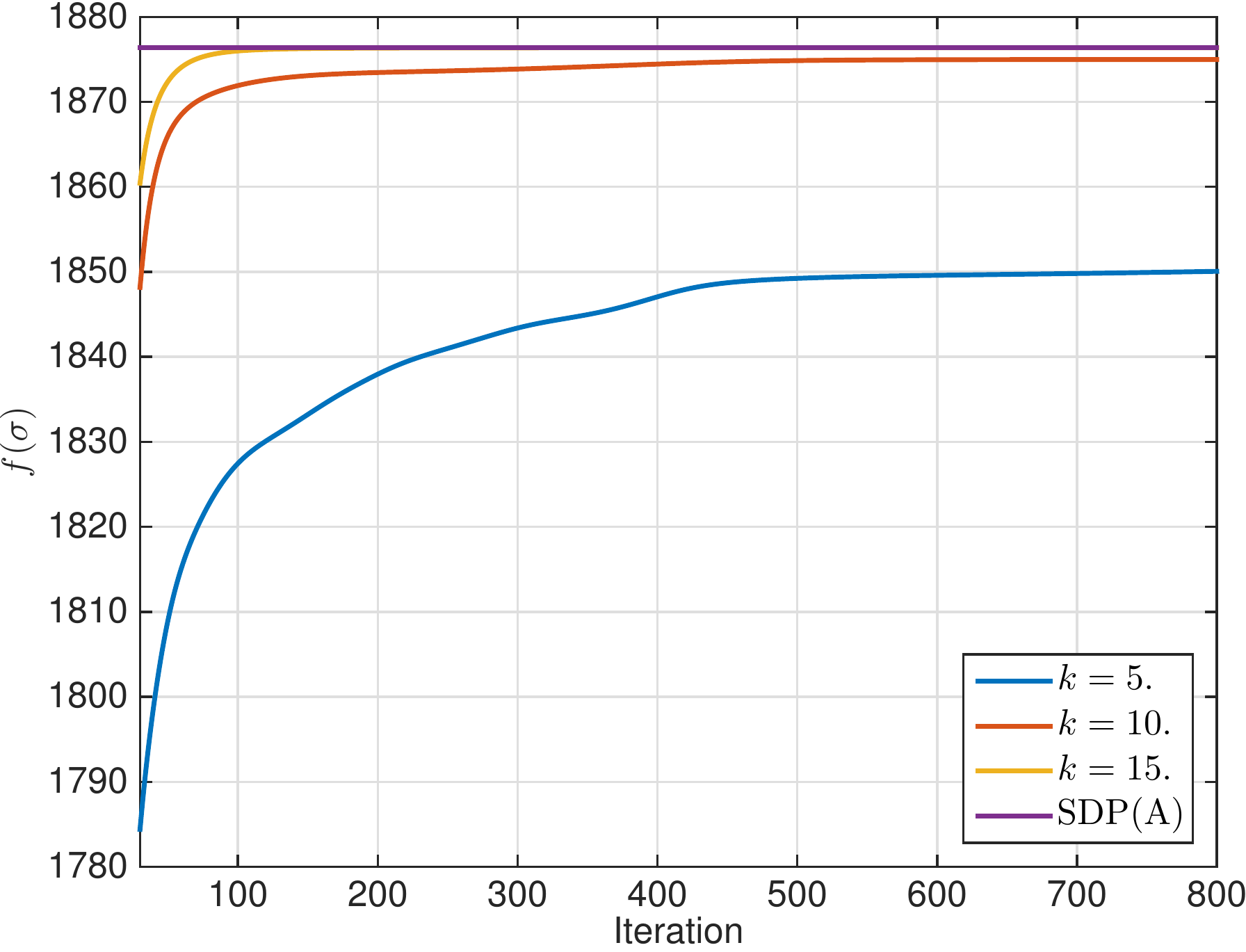}
%\captionsetup{labelformat=empty}
%\caption{}
(a)
\end{minipage}
\begin{minipage}{0.45\linewidth}
\centering
\includegraphics[width=0.95\textwidth]{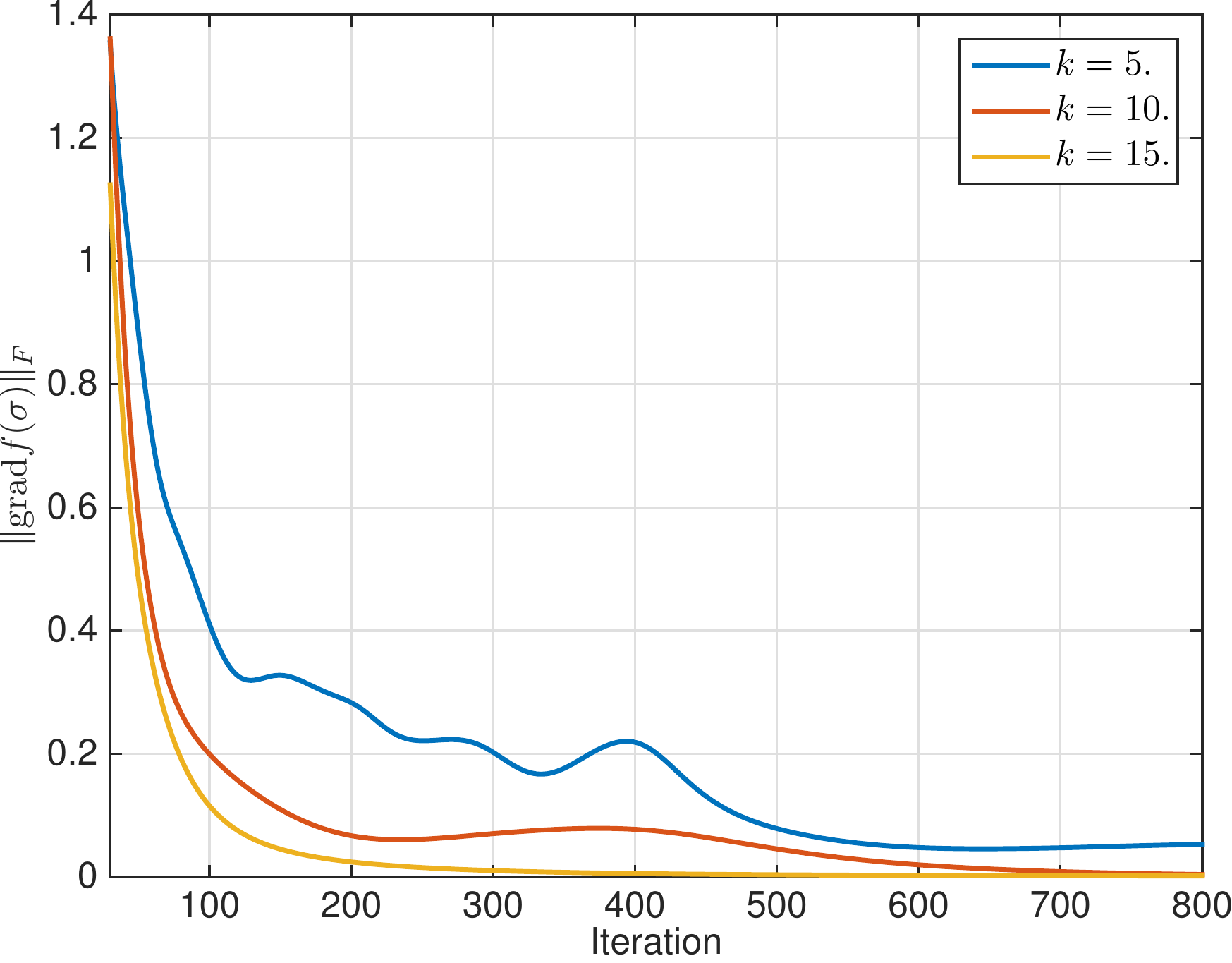}
%\captionsetup{labelformat=empty}
%\caption{}
(b)
\end{minipage}
\caption{Projected gradient ascent algorithm to optimize (\ref{SDP}) with $A \sim \text{GOE}(1000)$:  (a) $\func( \sigma)$ as a function of the iteration number for
a single realization of the trajectory; (b) $\Vert \tgrad \func( \sigma) \Vert_F$ as a function of the  iteration number. }\label{fig:convergence_gd}
\end{figure}

In this section we carry out some numerical experiments to illustrate our results. We also find interesting phenomena which are not captured by our analysis. 

Although Theorem \ref{thm:trustregion} provides a complexity bound for the Riemannian trust-region method (RTR), we observe that (projected) gradient ascent also converges very fast. That is, gradient ascent 
rapidly increases the objective function, is not trapped at a saddle point, and converges to a local maximizer eventually. In Figure \ref{fig:convergence_gd}, we take $A \sim \GOE(1000)$, and  use projected gradient ascent 
to solve the optimization problem (\ref{NcvxSDP}) with a random initialization and fixed step size. Figure \ref{fig:convergence_gd}a shows that the objective function increases
rapidly and converges within a small interval from the local maximum (which is upper bounded by the value $\SDP(A)$). Also the gap between the value obtained by this procedure and the value $\SDP(A)$ 
decreases rapidly with $k$.
 Figure \ref{fig:convergence_gd}b shows that the Riemannian gradient decreases very rapidly, but presents some non-monotonicity. We believe these bumps occur when the iterates 
are close to saddle points.

\begin{figure}[ht]
\centering
\begin{minipage}{0.45\linewidth}
\centering
\includegraphics[width=0.95\textwidth]{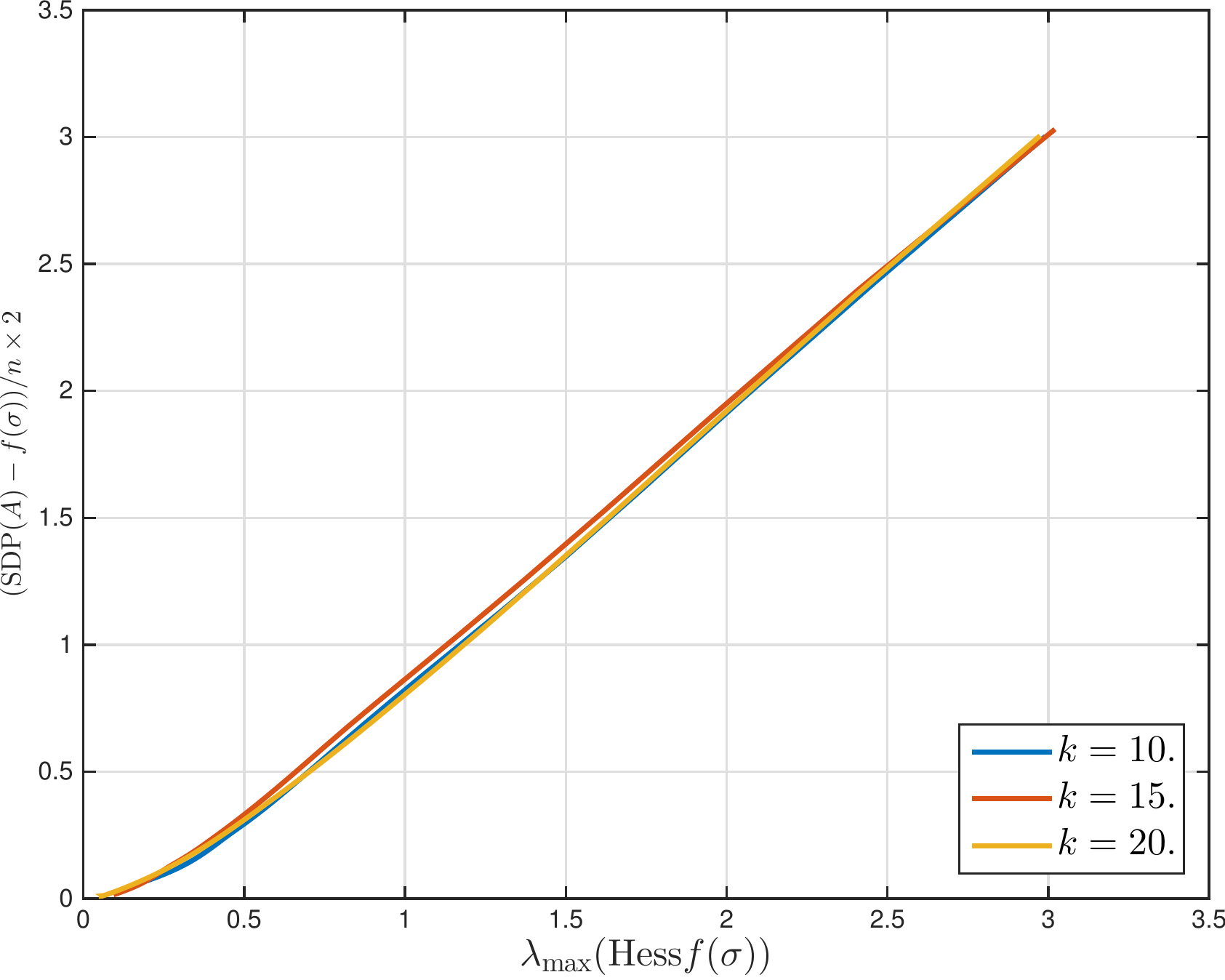}
%\caption{}\label{fig:epsilon}
(a)
\end{minipage}
\begin{minipage}{0.45\linewidth}
\centering
\includegraphics[width=0.95\textwidth]{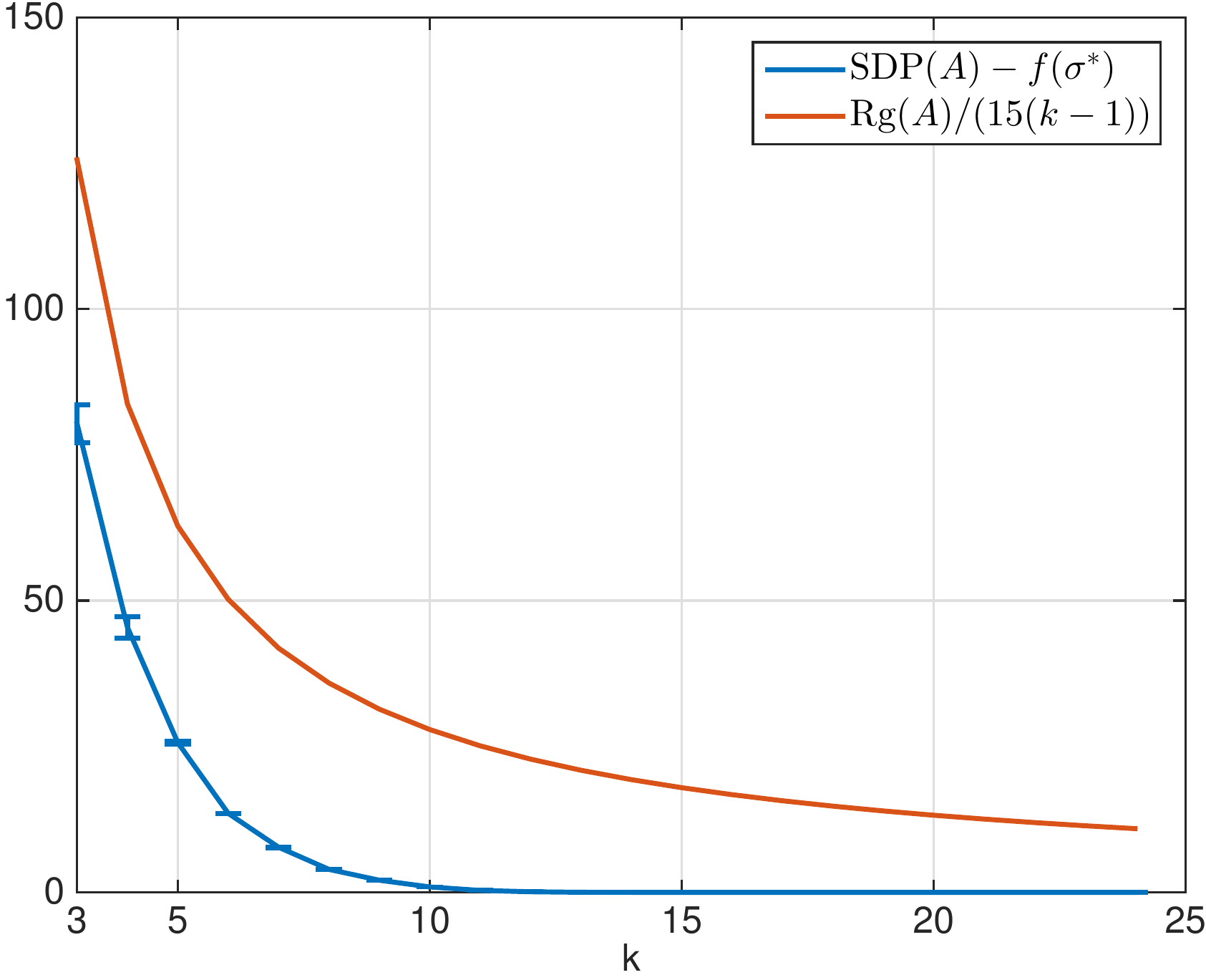}
%\caption{}\label{fig:gap}
(b)
\end{minipage}
\caption{Geometric properties of the rank-$k$ non-convex SDP, where $A \sim \text{GOE}(1000)$. (a). $\lambda_{\max}(\thess \func( \sigma))$ versus $(\SDP(A) - \func( \sigma))/n\times 2$. (b). $\SDP(A) - \func( \sigma^*)$ for different $k$, where $\sigma^* \in \cM_k$ is a local maximizer. }\label{fig:geometry}
\end{figure}

In Figure \ref{fig:geometry}, we examine some geometric properties of the rank-$k$ non-convex SDP. As above, we explore the landscape of this problem by projected gradient ascent.
 In Figure \ref{fig:geometry}a, we plot the curvature $\lambda_{\max}(\thess \func( \sigma))$ versus the gap from the SDP value $(\SDP(A) - \func( \sigma))/n\times 2$ along the iterations. 
When $\func( \sigma)$ is far from $\SDP(A)$, there is a linear relationship between these two quantities, which is consistent with Theorem \ref{thm:apprxconcavepoint}. In Figure \ref{fig:geometry}b, we plot the gap between 
$\SDP(A)$ and $\func( \sigma^*)$ for a local maximizer $\sigma^* \in \cM_k$ that is produced by projected gradient ascent, for different values of $k$. 
These data are averaged over $10$ realizations of the random matrix $A$. 
This gap converges to zero as $k$ gets large, and is upper bounded by the curve $\RGD(A)/(15 (k-1))$. This coincides with Theorem \ref{thm:apprxconcavepoint}, which predicts that this gap must be smaller than $\RGD(A)/(k-1)$. 
Note however that --in this case-- Theorem \ref{thm:apprxconcavepoint} is overly pessimistic, and the gap appears to decrease very rapidly with $k$.

\begin{figure}[ht]
\centering
\includegraphics[width=0.45\textwidth]{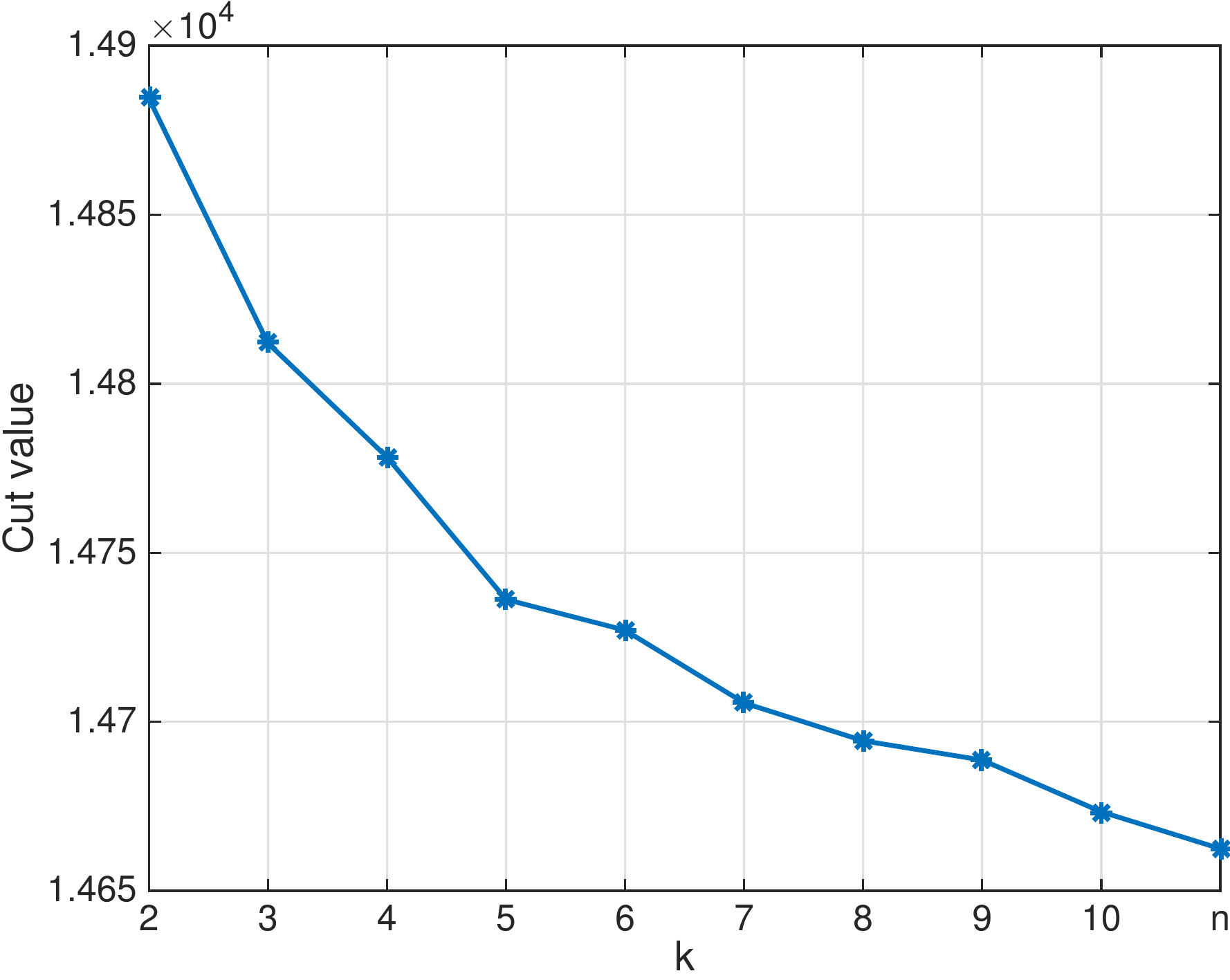}
\caption{Cut value found by rounding local maximizer of rank-$k$ non-convex SDP, for Erd\H{o}s-R\'{e}nyi random graphs with $n=1000$ and average degree $d=50$. Data are averaged over $10$ realizations. }\label{fig:MaxCut}
\end{figure}
Now we turn to study the MaxCut problem. Note that Theorem \ref{thm:MaxCut} gives a guarantee for the approximation ratio for the cut induced by any local maximizer of the rank-$k$ non-convex SDP (\ref{NcvxSDP}). 
In Figure \ref{fig:MaxCut}, we take the graph to be an Erd\H{o}s-R\'{e}nyi  graph with $n = 1000$ and average degree $d = 50$. We plot the cut value found by rounding the maximizer of the rank-$k$ non-convex SDP, 
for $k$ from $2$ to $10$, and also for $k = n$ which corresponds to the  (\ref{SDP}). Surprisingly, the cut value found by solving rank-$k$ non-convex problem is typically bigger than the cut value found by solving the original SDP. 
This provides a further reason to adopt the non-convex approach (\ref{NcvxSDP}). It appears to provide a significantly tight relaxation for random instances.

\begin{figure}[ht]
\centering
\begin{minipage}{0.45\linewidth}
\centering
\includegraphics[width=0.95\textwidth]{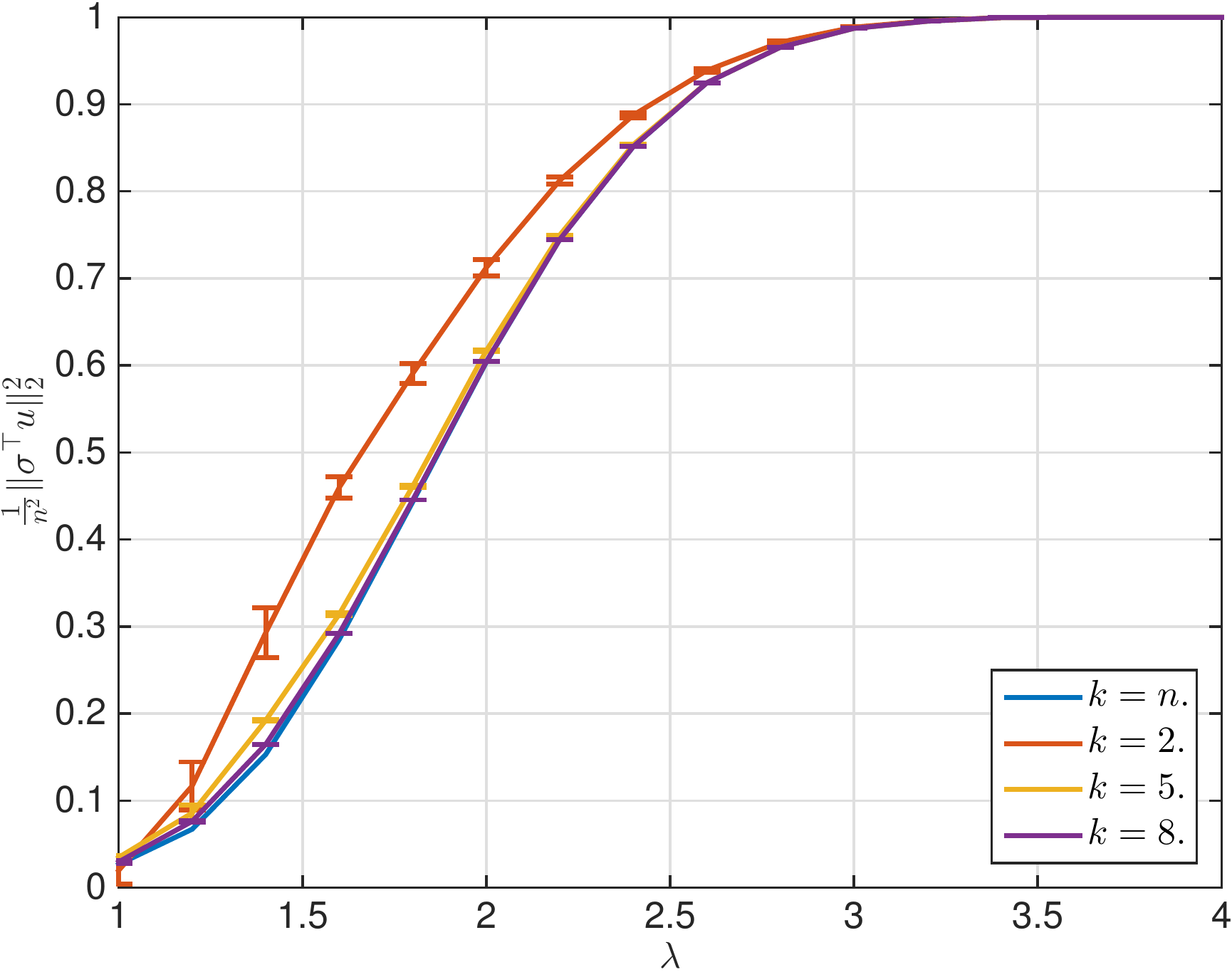}
(a)
%\caption{}
\end{minipage}
\begin{minipage}{0.45\linewidth}
\centering
\includegraphics[width=0.95\textwidth]{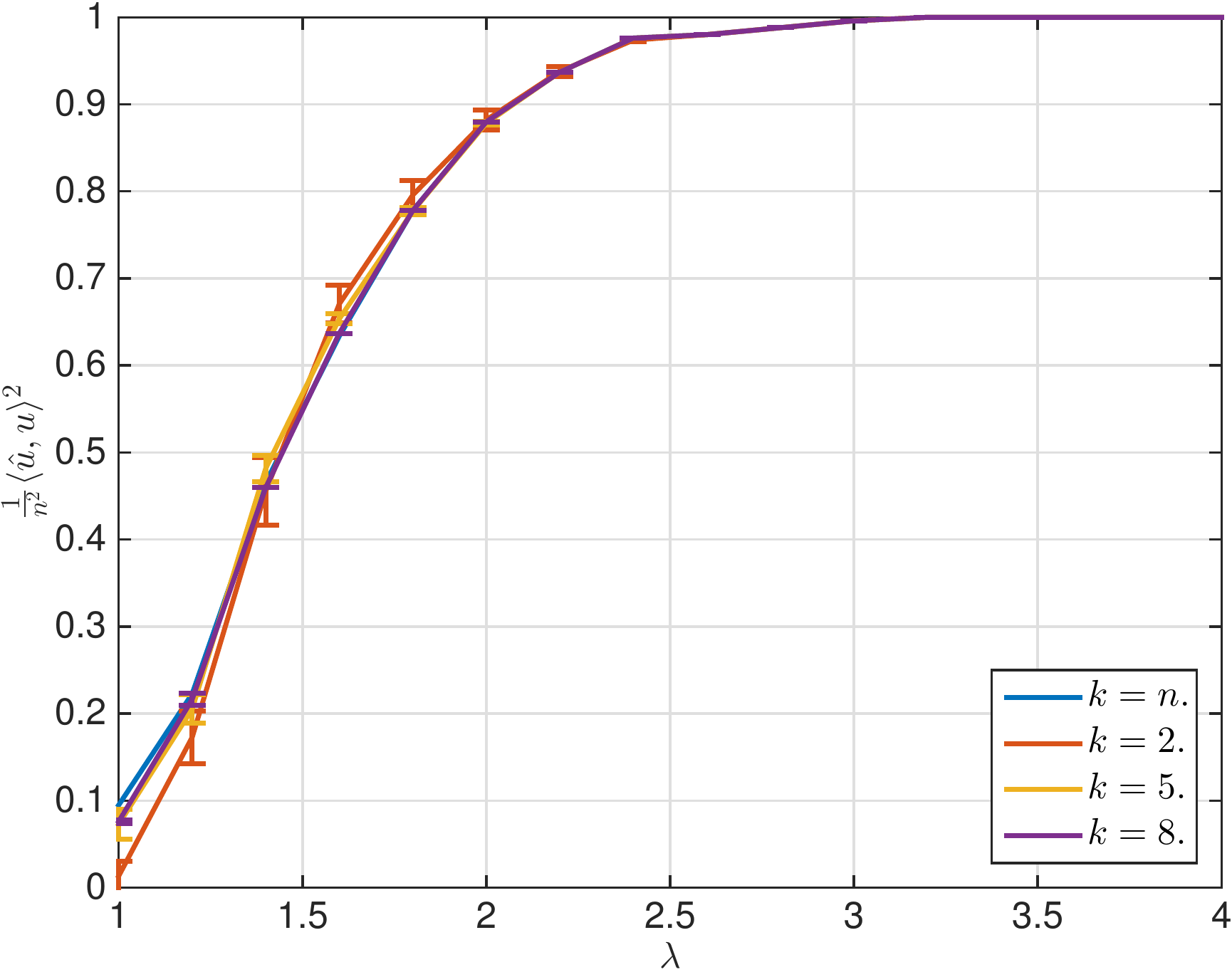}
%\caption{}
(b)
\end{minipage}
\caption{$\mZt$ synchronization: correlation between estimator and ground truth $\Vert \sigma^\sT u \Vert_2^2/n^2$ and $\<\hat u^\sT ,u \>^2/n$ versus $\lambda$. }\label{fig:z2sync}
\end{figure}

In order to study $\mZt$ synchronization, we consider the matrix $A = (\lambda/n) u u^\sT + W_n$ where $W_n \sim \text{GOE}(n)$ for $n = 1000$. Figure \ref{fig:z2sync}a shows the correlation $\Vert \sigma^\sT u \Vert_2^2/n^2$ of a 
local maximizer $\sigma \in \cM_k$ produced by projected gradient ascent, with the ground truth $u$. In  Figure \ref{fig:z2sync}b 
we construct label estimates $\hu(A)  = \sign(v_1(\sigma))$ where $v_1(\sigma)$ is the principal left singular vector of $\sigma\in\reals^{n\times k}$.
We plot  the correlation $(\langle \hat u, u\rangle/n)^2$ as a function of $\lambda$. In both cases, results are averaged over $10$ realizations of the matrix $A$.
Surprisingly, the resulting correlation is strongly concentrated, despite  the fact that gradient ascent converges to a random local maximum  $\sigma \in \cM_k$.

\begin{figure}[ht]
\centering
\includegraphics[width=0.45\textwidth]{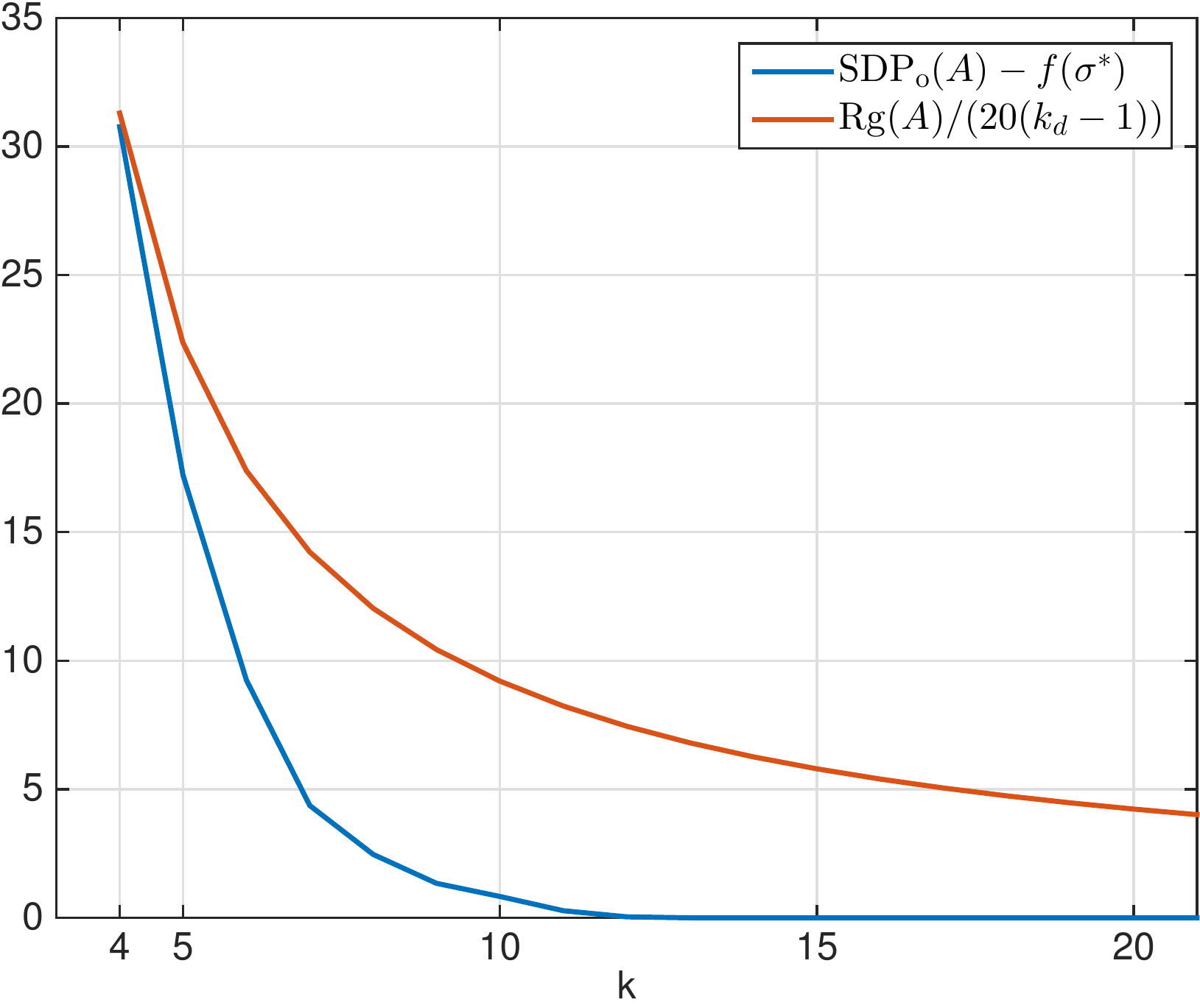}
\caption{$\SDP_o(A) - \func( \sigma^*)$ for different $k$, where $\sigma^* \in \cM_{o,d,k}$ is a local maximizer. }\label{fig:ocsdp}
\end{figure}

Finally, we turn to the $\SO(3)$ synchronization problem, and study the local maximizer of the Orthogonal-Cut SDP (\ref{OC-SDP}). We sample a matrix $A \sim \text{GOE}(300)$, and find the local maximum of the rank-$k$ non-convex Orthogonal-Cut SDP (\ref{Ncvx-OC-SDP}). In Figure \ref{fig:ocsdp} we plot the gap between $\SDP_o(A)$ and $\func( \sigma^*)$ for a local maximizer $\sigma^* \in \R^{n\times k}$ produced by
projected gradient ascent for different $k$. This gap converges to zero as $k$ is larger, and is upper bounded by  $\RGD(A)/(20 (k_d-1))$. This is in agreement with 
Theorem \ref{thm:SOd_grothendieck}, which predicts that the gap is smaller than $\RGD(A)/(k_d -1)$.

%\clearpage

\section{Other proofs}\label{sec:proofs}

\subsection{Proof of Theorem \ref{thm:MaxCut}}\label{sec:proof_thm3}

Note that problem (\ref{eq:NcvxSDP})  is equivalent to problem (\ref{NcvxSDP}) with matrix $A = -A_G$. Applying Theorem \ref{thm:apprxconcavepoint},
 and noting that the elements of $A_G$ are non-negative, we for any local maximizer $\sigma^*$ of the problem (\ref{eq:NcvxSDP}), 
and any $X^*$ optimal solution of the SDP (\ref{eq:SDPCut}),
\begin{equation}
\begin{aligned}
\langle \sigma^*, -A_G \sigma^* \rangle \geq & \langle -A_G, X^* \rangle - \frac{1}{k-1}(\langle -A_G, X^* \rangle + \SDP(A_G))\\
= & \langle -A_G, X^* \rangle - \frac{1}{k-1}(\langle -A_G, X^* \rangle + \sum_{i,j=1}^n A_{G,ij}).
\end{aligned}
\end{equation}

Thus, we have
\begin{equation}
\begin{aligned}
&\frac{1}{4} \sum_{i,j=1}^n A_{G,ij} (1-\langle \sigma_i^*, \sigma_j^*\rangle) = \frac{1}{4} \sum_{i,j=1}^n A_{G,ij} + \frac{1}{4} \langle \sigma^*, -A_G \sigma^*\rangle \\
\ge & \frac{1}{4} \sum_{i,j=1}^n A_{G,ij} + \frac{1}{4}[\langle -A_G, X^* \rangle - \frac{1}{k-1}(\langle -A_G, X^* \rangle + \sum_{i,j=1}^n A_{G,ij})]\\
=& \left(1-\frac{1}{k-1}\right) \times \frac{1}{4} \sum_{i,j=1}^n A_{G,ij}(1 - X_{ij}^*) = \left(1-\frac{1}{k-1}\right) \times \text{SDPCut}(G) \\
\geq & \left(1-\frac{1}{k-1}\right)\times \text{MaxCut}(G). 
\end{aligned}
\end{equation}

Applying the randomized rounding scheme of \cite{goemans1995improved}, we sample a vector $u \sim \normal (\bzero, \id_k)$, and define $v \in \{\pm 1\}^n$ by
 $v_i = \sign(\langle \sigma_i^*, u\rangle)$, then we obtain
\[
\E\left[ \frac{1}{4} \sum_{i,j=1}^n A_{G,ij}(1-\langle v_i, v_j \rangle)\right] \geq \alpha_* \times \frac{1}{4} \sum_{i,j=1}^n A_{G,ij}(1-\langle \sigma_i^*, \sigma_j^* \rangle) \geq \alpha_* \times \left(1-\frac{1}{k-1}\right) \times \text{MaxCut}(G).
\]
Therefore, for any local maximizer $\sigma^*$, it gives an $\alpha_* \times (1-1/(k-1))$-approximate solution of the MaxCut problem. 

If $\sigma^*$ is an $\eps = 2\RGD(A_G) / (n (k-1))$-approximate concave point, using Theorem \ref{thm:apprxconcavepoint} and the same argument, we can prove that it gives an $\alpha_* \times (1-2/(k-1))$-approximate solution of the MaxCut problem.

\subsection{Proof of Theorem \ref{thm:Z2sync1}}\label{sec:proof_thm4}

Let $A(\lambda) = \lambda/n \cdot u u^\sT + W_n$.
For any local maximum $\sigma \in \Crit_{n,k}$ of the rank-$k$ non-convex MaxCut SDP problem, according to Theorem \ref{thm:apprxconcavepoint}, we have
\[
\func( \sigma) \geq \SDP(A(\lambda)) - \frac{1}{k-1}(\SDP(A(\lambda)) + \SDP(-A(\lambda))).
\]
Therefore
\begin{equation}\label{eq:nonvanishing1}
\begin{aligned}
\frac{\lambda}{n} \Vert \sigma^\sT u \Vert_2^2 \geq& \left(1-\frac{1}{k-1}\right) \SDP \left(\frac{\lambda}{n} u u^\sT + W_n \right) - \frac{1}{k-1} \SDP \left(-\frac{\lambda}{n} u u^\sT - W_n \right) - \left\langle W_n, \sigma \sigma^\sT \right\rangle\\
\geq & \left(1-\frac{1}{k-1}\right) \SDP\left(\frac{\lambda}{n} u u^\sT + W_n\right) - \frac{1}{k-1} \SDP\left( - W_n\right) - \SDP\left(W_n\right).
\end{aligned}
\end{equation}

Using the convergence of the SDP value as proved in \cite[Theorem 5]{montanari2016semidefinite}, for any $\lambda > 1$, there exists $\Delta (\lambda) >0$ such that, for any $\delta > 0$,
the following holds with high probability
\[
 \frac{1}{n}\SDP \left( \pm W_n \right)  \leq 2 + \delta, \quad \text{and} \quad 
 \frac{1}{n}\SDP \left( \frac{\lambda}{n} u u^\sT + W_n \right)  \geq 2 + \Delta(\lambda)\, .
\]

Therefore, we have with high probability
\begin{equation}\label{eq:an_equation_which_will_be_cited}
\begin{aligned}
\frac{1}{n^2} \Vert \sigma^\sT u \Vert_2^2 \geq& \frac{1}{\lambda} \left[ \big(2+\Delta(\lambda)\big) \times\Big(1-\frac{1}{k-1}\Big) - \left(1 + \frac{1}{k-1}\right) \times (2+\delta)\right] \\
=&  \left(1-\frac{1}{k-1}\right) \frac{\Delta(\lambda)}{\lambda}  - \frac{4 + \delta}{k-1} \cdot \frac{1}{\lambda} - \frac{\delta}{\lambda}. 
\end{aligned}
\end{equation}
Since $\Delta(\lambda) > 0$ for $\lambda >1$, there exists a $k_*(\lambda)$ such that the above expression is greater than $\eps$ for sufficiently small $\eps$ and $\delta$, which concludes the proof.

\subsection{Proof of Theorem \ref{thm:correlationlimit}}\label{sec:proof_thm7}

We decompose the proof into two parts. In part $(a)$, we prove that almost surely
\[
\liminf_{n \rightarrow \infty} \inf_{\sigma \in \C_{n,k}} \frac{1}{n^2}\Vert \sigma^\sT u \Vert_2^2 \geq 1 - \frac{1}{k} - \frac{4}{\lambda},
\]
using only the second order optimality condition. In part $(b)$, we incorporate the first order optimality condition and prove that as $\lambda \geq 12 k$, we have almost surely
\[
\liminf_{n \rightarrow \infty} \inf_{\sigma \in \C_{n,k}} \frac{1}{n^2}\Vert \sigma^\sT u \Vert_2^2 \geq 1 - \frac{16}{\lambda}. 
\]

\begin{proof}

\subsubsection*{Part $(a)$}

The proof of this part is similar to the proof of Theorem \ref{thm:apprxconcavepoint}. We replace the matrix $A$ by the expression $A = u u^\sT + \Delta$, where $u \in \{ \pm 1 \}^n$ and $\Delta = n/\lambda \cdot W_n$. Let $g \in \mathbb R^{k}$, $g \sim \normal(\bzero, 1/k \cdot \id_k)$, and $W = [\proj_1^\perp g u_1,  \ldots, \proj_n^\perp g u_n]^\sT \in T_{\sigma}\cM_k$, where $\proj_i^\perp = \id_k - \sigma_i \sigma_i^\sT \in \R^{k \times k}$. Due to the second order optimality condition, similar to the calculation in Theorem \ref{thm:apprxconcavepoint}, we have for any local maximizer $\sigma$ of the rank-$k$ non-convex SDP problem:
\[
0 \leq \E_g \langle W, (\Lambda(\sigma) - A) W \rangle 
= \left(1-\frac{1}{k}\right)f(\sigma) - \left(1-\frac{2}{k} \right) \sum_{i,j=1}^n A_{ij} u_i u_j - \frac{1}{k}\sum_{i,j=1}^n A_{ij} u_i u_j \langle \sigma_i, \sigma_j \rangle^2. 
\]
Plugging in the expression of $A$, we obtain
\[
\left(1- \frac{1}{k}\right) \langle u u^\sT + \Delta, \sigma \sigma^\sT \rangle - \left(1-\frac{2}{k}\right)(n^2 + \langle u, \Delta u\rangle) - \frac{1}{k}\sum_{i,j=1}^n \left[ \langle \sigma_i, \sigma_j \rangle^2 + \Delta_{ij} u_i u_j \langle \sigma_i, \sigma_j\rangle^2 \right] \geq 0.
\]
Letting $\barX_{ij} = u_i u_j \langle \sigma_i, \sigma_j \rangle^2$, we have
\[
\left(1-\frac{1}{k}\right)\Vert \sigma^\sT u \Vert_2^2 \geq \left(1-\frac{2}{k}\right) n^2  - \left(1-\frac{1}{k}\right) \langle \sigma, \Delta \sigma\rangle + \left(1-\frac{2}{k}\right) \langle u, \Delta u \rangle + \frac{1}{k} \sum_{i,j=1}^n \left[ \langle \sigma_i, \sigma_j \rangle^2 + \Delta_{ij} \barX_{ij} \right]. 
\]
Recall that $\text{rank}(\sigma \sigma^\sT)  = k$, and $\tr(\sigma \sigma^\sT) = n$. Thus, we get the lower bound
\[
\begin{aligned}
&\sum_{i,j=1}^n \langle \sigma_i, \sigma_j \rangle^2 = \Vert \sigma \sigma^\sT\Vert_F^2  = \sum_{i=1}^k \lambda_i^2(\sigma \sigma^\sT)\ge \frac{1}{k} \left(\sum_{i=1}^k \lambda_i(\sigma \sigma^\sT)\right)^2 = \frac{1}{k}\left(\tr(\sigma \sigma^\sT) \right)^2 =  \frac{n^2}{k}.
\end{aligned}
\]
Also note that $\barX$ is a feasible point of (\ref{SDP}). Therefore,
\[
\begin{aligned}
\left(1-\frac{1}{k}\right)\Vert \sigma^\sT u \Vert_2^2 \geq& \left(1-\frac{2}{k}  + \frac{1}{k^2}\right) n^2  - \left(1-\frac{1}{k}\right) \langle \sigma, \Delta \sigma\rangle + \left(1-\frac{2}{k}\right) \langle u, \Delta u \rangle + \frac{1}{k} \langle \Delta, \barX\rangle\\
 \geq  & \left(1-\frac{1}{k}\right)^2 n^2  - \left(1-\frac{1}{k}\right) \SDP(\Delta) - \left(1-\frac{1}{k}\right) \SDP(-\Delta) \\
 \geq & \left(1-\frac{1}{k}\right)^2 n^2  - 2 \left(1-\frac{1}{k}\right) n \Vert \Delta \Vert_{\op} \\
\end{aligned}
\]
which implies that
\begin{align}\label{eq:clb1}
\liminf_{n \rightarrow \infty} \inf_{\sigma \in \C_{n,k}} \frac{1}{n^2}\Vert \sigma^\sT u \Vert_2^2
 \geq \liminf_{n \rightarrow \infty} (1-\frac{1}{k} - \frac{2}{\lambda} \Vert W_n \Vert_{\op})
 = 1 - \frac{1}{k} - \frac{4}{\lambda}, \quad a.s.
\end{align}
where we used the fact that for a GOE matrix $W_n$, we have $\lim_{n\rightarrow \infty} \Vert W_n \Vert_{\op} = 2$ almost surely \cite{anderson2010introduction}.

\subsubsection*{Part $(b)$}

In part $(a)$ we only used the second order optimality condition. In this part of the proof, we will incorporate the first order optimality condition. Note that as $\lambda < 12 k$, the bound in part $(a)$ is better. So in this part, we only consider the case when $\lambda \geq 12 k$. 

Without loss of generality, let $u = \ones$, the vector with all entries equal to one. Let $\sigma \in \R^{n\times k}$ be a local optimizer of the rank-$k$ non-convex SDP problem. We remark that the cost function is invariant by a right rotation of $\sigma$. We can therefore assume that $\sigma = (v_1, \ldots, v_k)$ where $v_i \in \R^{n}$ and $\langle v_i, v_j \rangle= 0$ for $i\neq j$ (take the SVD decomposition $\sigma = U \Sigma V^\sT$ and consider $\tilde{\sigma} = U \Sigma$). Let $X = \sigma \sigma^\sT$ and $A(\lambda) = (\lambda/n) \cdot \ones \ones^\sT + W_n$. For simplicity, we will sometimes omit the dependence on $\lambda$ and write $A = A(\lambda)$.

We decompose the proof into the following steps. 

\medskip
\noindent {\bf Step 1} \textit{Upper bound on $\langle \ones, v_j \rangle^2/n^2$, for $j=2,\ldots, k$, using the first order optimality condition.}

The first order optimality condition gives $A \sigma  = \ddiag(A \sigma \sigma^\sT) \, \sigma$, which implies that
\[
(A v_i) \circ v_j = (A v_j) \circ v_i,
\]
for any $i \neq j$, where we denoted $u \circ v$ the entry-wise product of $u$ and $v$. Replacing $A$ by its expression gives
\[
\left(\left(\frac{\lambda}{n} \ones \ones^\sT + W_n \right) v_i \right)\circ v_j = \left( \left(\frac{\lambda}{n} \ones \ones^\sT + W_n \right) v_j \right)\circ v_i,
\]
which implies
\[
\langle \ones, v_i\rangle v_j - \langle \ones, v_j \rangle v_i = \frac{n}{\lambda} \left[-\left(W_n v_i \right)\circ v_j + \left(W_n v_j \right)\circ v_i\right].
\]
We take the norm of this expression and, recalling that $\langle v_i, v_j \rangle = 0$, we obtain
\begin{equation}\label{eq:foc0}
\begin{aligned}
\langle \ones, v_i \rangle^2 \Vert v_j \Vert_2^2 + \langle \ones, v_j \rangle^2 \Vert v_i \Vert_2^2 \leq& \frac{n^2}{\lambda^2} \left[\left\Vert( W_n v_i)\circ v_j \right\Vert_2+\left\Vert( W_n v_j)\circ v_i \right\Vert_2 \right]^2. \\
\end{aligned}
\end{equation}
Notice that $\Vert v_j\Vert_\infty \leq 1, \forall j\in [k]$, hence 
\begin{equation}\label{eq:foc1}
\begin{aligned}
\langle \ones, v_i \rangle^2 \Vert v_j \Vert_2^2 + \langle \ones, v_j \rangle^2 \Vert v_i \Vert_2^2 \leq & \frac{n^2}{\lambda^2} \left[\left\Vert W_n v_i \right\Vert_2+\left\Vert W_n v_j \right\Vert_2 \right]^2 \\
\leq &\frac{n^2}{\lambda^2}\left\Vert W_n \right\Vert_{\op}^2 \left[ \left \Vert v_i \right\Vert_2+\left\Vert v_j \right\Vert_2 \right]^2 \\
\leq & \frac{2n^2}{\lambda^2} \left\Vert W_n \right\Vert_{\op}^2 \left(\Vert v_i \Vert_2^2 + \Vert v_j \Vert_2^2\right). 
\end{aligned}
\end{equation}

Without loss of generality, let us assume that $\Vert v_1 \Vert_2 \geq \Vert v_j \Vert_2$ for $j \geq 2$ which implies 
\[
\langle \ones, v_j \rangle^2 \Vert v_1 \Vert_2^2 \leq \frac{4n^2}{\lambda^2} \left\Vert W_n \right\Vert_{\op}^2 \Vert v_1 \Vert_2^2, \quad \text{for } j \geq 2.
\]
We deduce the following upper bound
\begin{align}\label{eq:ub1}
\limsup_{n \rightarrow \infty} \sup_{\sigma \in \Crit_{n,k}} \frac{1}{n^2}\langle \ones, v_j \rangle^2 \leq \frac{16}{\lambda^2}, \quad a.s.
\end{align}
for $j = 2,\ldots, k$, where we use the fact that for a GOE matrix $W_n$, we have $\lim_{n\rightarrow \infty} \Vert W_n \Vert_{\op} = 2$ almost surely.

\medskip
\noindent {\bf Step 2} \textit{Lower bound on $\langle \ones, v_1 \rangle^2/n^2$.} 

We combine equation (\ref{eq:clb1}) and (\ref{eq:ub1}) to get almost surely
\[
\liminf_{n \rightarrow \infty} \inf_{\sigma \in \Crit_{n,k}} \frac{1}{n^2}\langle \ones, v_1 \rangle^2 = \liminf_{n \rightarrow \infty} \inf_{\sigma \in \Crit_{n,k}} \left[ \frac{1}{n^2} \Vert \sigma^\sT \ones \Vert_2^2 - \frac{1}{n^2} \sum_{j=2}^{k} \langle \ones, v_j \rangle^2 \right]
\geq  1 - \frac{1}{k} - \frac{4}{ \lambda} - \frac{16k}{\lambda^2}.
\]
Since we assumed that $\lambda \geq 12 k$ and $k 
\geq 2$, we obtain, almost surely,
\begin{align}\label{eq:lb2}
\liminf_{n \rightarrow \infty} \inf_{\sigma \in \Crit_{n,k}} \frac{1}{n^2}\langle \ones, v_1 \rangle^2 \geq 1 - \frac{1}{k} - \frac{4}{12 k} - \frac{16k}{144 k^2} \geq \frac{1}{4}.
\end{align}
The second inequality above is loose but it is sufficient for our purposes. 

\medskip
\noindent {\bf Step 3} \textit{Upper bound on $\Vert v_a \Vert_2^2$ for $a \in \{ 2,\ldots, k \}$.} 

In Equation (\ref{eq:foc1}), let us take $i = 1$ and $j = a \in \{ 2,\ldots, k \}$, we have
\begin{equation}\label{eq:ub3}
\begin{aligned}
&\langle \ones, v_1 \rangle^2 
 \Vert v_a \Vert_2^2 +  \langle \ones, v_a \rangle^2 \Vert v_1 \Vert_2^2 \leq \frac{2n^2}{\lambda^2} \left\Vert W_n \right\Vert_{\op}^2 \left(\Vert v_1 \Vert_2^2 + \Vert v_a \Vert_2^2\right) \leq  \frac{2 n^3}{\lambda^2} \left\Vert W_n \right\Vert_{\op}^2. 
\end{aligned}
\end{equation}

Combining equation (\ref{eq:ub3}) and (\ref{eq:lb2}) results in the following upper bound for $\lambda \geq 12 k$, 
\begin{align}\label{eq:ub2}
\limsup_{n \rightarrow \infty} \sup_{\sigma \in \Crit_{n,k}}  \frac{1}{n} \Vert v_a \Vert_2^2 \leq \frac{32}{\lambda^2},
\end{align}
holding almost surely for any $a \in \{ 2,\ldots, k \}$.

\medskip
\noindent {\bf Step 4} \textit{Lower bound on $f(\sigma)$.}  

By second order optimality of $\sigma$, for any vectors $\{ \xi_i \}_{i=1}^n$ satisfying $\<\sigma_i, \xi_i\>=0$, we have $\langle \xi, (\Lambda-A) \xi \rangle \geq 0$ where $\xi = [\xi_1,\ldots, \xi_n]^\sT$ and $\Lambda = \ddiag(A \sigma \sigma^\sT)$. Take $\xi_i = e_a - \<\sigma_i,e_a\> \sigma_i$, where $e_a$ is the $a$-th canonical basis vector in $\R^{k}$, $a\in\{2,\dots,k\}$. Noting that $\sigma = (\sigma_1,\ldots, \sigma_n)^\sT = (v_1,\ldots, v_k)$, we have $\<\sigma_i, e_a\>= v_{a,i}$. Therefore, we have
\begin{align}
\<\xi_i,\xi_j\> = 1- v_{a,i}^2-v_{a,j}^2 + X_{ij}v_{a,i}v_{a,j}\,.
\end{align}
Using the second order stationarity condition with this choice of $\xi_i$, we have
\[
\begin{aligned}
0 \leq& \sum_{i,j=1}^n (\Lambda - A)_{ij} \langle \xi_i, \xi_j\rangle \\
=& \sum_{i,j=1}^n (\Lambda - A)_{ij}(1- v_{a,i}^2-v_{a,j}^2 + X_{ij}v_{a,i}v_{a,j})\\
=& \sum_{i=1}^n \Lambda_{ii} (1 - v_{a,i}^2) - \sum_{i,j=1}^n A_{ij}(1- 2 v_{a,i}^2+ X_{ij}v_{a,i}v_{a,j}),
\end{aligned}
\]
which implies
\begin{equation}\label{eq:thm_eq1}
\begin{aligned}
f(\sigma) =& \tr(\Lambda) \geq  \sum_{i=1}^n \Lambda_{ii} v_{a,i}^2 + \sum_{i,j=1}^n A_{ij}(1- 2 v_{a,i}^2+ X_{ij}v_{a,i}v_{a,j}) \\
=& \langle \ones, A \ones \rangle + \sum_{i=1}^n \Lambda_{ii} v_{a,i}^2 - 2 \sum_{i,j=1}^n A_{ij} v_{a,i}^2 + \sum_{i,j=1}^n A_{ij} X_{ij} v_{a,i} v_{a,j}\\
\equiv & \langle \ones, A \ones \rangle + B_1+B_2+B_3.
\end{aligned}
\end{equation}

Consider the first term $B_1$. It is easy to see that the second order stationary condition implies $(\Lambda - A)_{ii} \geq 0$. Thus, we have
\[
\begin{aligned}
B_1 =& \sum_{i=1}^n \Lambda_{ii} v_{a,i}^2 \geq \sum_{i=1}^n A_{ii} v_{a,i}^2 = \sum_{i=1}^n (\lambda/n + W_{n, ii}) v_{a,i}^2 \\
\geq & \sum_{i=1}^n W_{n, ii} v_{a,i}^2 \geq - \max_{i \in [n]} \vert W_{n, ii}\vert \cdot \sum_{i=1}^n v_{a,i}^2 \\
\geq & -\Vert W_n \Vert_{\op} \Vert v_{a}\Vert_2^2.
\end{aligned}
\]

Next consider the second term $B_2$. We have 
\[
\begin{aligned}
\vert B_2 \vert =& 2 \vert \langle \ones, A ( v_{a} \circ v_a ) \rangle \vert =  2 \vert \langle \ones, (\lambda/n \cdot \ones \ones^\sT + W_n) ( v_{a} \circ v_a ) \rangle \vert \\
\leq & 2 \lambda \Vert v_{a} \Vert_2^2 + 2 \vert \langle \ones, W_n (v_a \circ v_a) \rangle \vert \leq 2 \lambda \Vert v_{a} \Vert_2^2 + 2 \sqrt n \Vert W_n \Vert_{\op} \Vert v_a \circ v_a \Vert_2 \\
\leq & 2 \lambda \Vert v_{a} \Vert_2^2 + 2 \sqrt n \Vert W_n \Vert_{\op} \Vert v_a \Vert_2. 
\end{aligned}
\]
where the last inequality is because $\vert v_{a,i} \vert \leq 1$ so that $\Vert v_a \circ v_a \Vert_2\leq \Vert v_a \Vert_2$. 

Finally, consider the last term $B_3$. 
\[
\begin{aligned}
B_3 =& \langle v_a, ((\lambda/n \cdot \ones \ones^\sT + W_n) \circ X) v_a\rangle\\
=& \lambda/n \cdot \langle v_a, X v_a \rangle + \langle v_a, (W_n \circ X) v_a \rangle \\
\geq & \langle v_a, (W_n \circ X) v_a \rangle \geq -\Vert W_n \circ X \Vert_{\op} \Vert v_a \Vert_2^2\\
\geq & - \Vert W_n \Vert_{\op} \Vert v_a \Vert_2^2,
\end{aligned}
\]
where the last inequality used a fact that if $X \in \R^{n \times n}$ is in the elliptope, we have $\Vert W \circ X \Vert_{\op} \leq \Vert W \Vert_{\op}$ for any $W \in \R^{n\times n}$.

Here is the justification of the above fact. For $X$ in the elliptope, we have $X_{ii} = 1$ and $X \succeq 0$. For any $Z$ satisfying $Z \succeq 0$ and $\tr(Z) \leq 1$, $X \circ Z$ also satisfies $X \circ Z \succeq 0$ and $\tr(X \circ Z) \leq 1$. Therefore, using the variational representation of the operator norm, we have
\[
\begin{aligned}
\Vert W \circ X \Vert_{\op} =& \max\left \{ \sup_{Z \succeq 0, \tr(Z) \leq 1} \langle W \circ X, Z \rangle, \sup_{Z \succeq 0, \tr(Z) \leq 1} \langle - W \circ X, Z \rangle \right\}\\
=& \max\left \{ \sup_{Z \succeq 0, \tr(Z) \leq 1} \langle W,  X \circ Z \rangle, \sup_{Z \succeq 0, \tr(Z) \leq 1} \langle - W, X \circ Z \rangle \right\}\\
\leq & \max\left \{ \sup_{Y \succeq 0, \tr(Y) \leq 1} \langle W,  Y \rangle, \sup_{Y \succeq 0, \tr(Y) \leq 1} \langle - W, Y \rangle \right\} = \Vert W \Vert_{\op}.
\end{aligned}
\]

\medskip
\noindent {\bf Step 5} \textit{Finish the proof.}  

Noting that $f(\sigma) = \lambda/n \cdot \Vert \sigma^\sT \ones \Vert_2^2 + \langle \sigma, W_n \sigma\rangle$ and $\langle \ones, A \ones\rangle = n \lambda + \langle \ones, W_n \ones \rangle$, we rewrite Equation (\ref{eq:thm_eq1}) as following
\[
\frac{1}{n^2}\Vert \sigma^\sT \ones \Vert_2^2 \geq 1 - \frac{1}{\lambda n}(\langle \sigma, W_n \sigma\rangle - \langle \ones, W_n \ones \rangle)  + \frac{1}{\lambda n}(B_1 + B_2 + B_3).
\]
Plug in the lower bound of $B_1$, $B_2$, $B_3$, we have almost surely
\[
\begin{aligned}
&\liminf_{n\rightarrow \infty} \inf_{\sigma \in \Crit_{n,k}}\frac{1}{n^2}\Vert \sigma^\sT \ones \Vert_2^2 \\
\geq&\liminf_{n\rightarrow \infty} \inf_{\sigma \in \Crit_{n,k}} \left\{ 1 - \frac{2}{\lambda}\Vert W_n \Vert_{\op}  - \frac{1}{\lambda n}(2 \Vert W_n \Vert_{\op} \Vert v_a \Vert_2^2 + 2\lambda \Vert v_a \Vert_2^2 + 2 \sqrt n \Vert W_n \Vert_{\op} \Vert v_a \Vert_2)\right\}\\
\geq & 1 - \frac{4}{\lambda}  - \frac{1}{\lambda }(2 \times 2 \times \frac{32}{\lambda^2} + 2 \lambda \times \frac{32}{\lambda^2} + 2 \times 2 \times \frac{\sqrt{32}}{\lambda})\\
\geq & 1 - \frac{16}{\lambda}.
\end{aligned}
\]
Here we used Equation (\ref{eq:ub2}), $\lambda \geq 12k \geq 24$, and the fact that for a GOE matrix $W_n$, we have $\lim_{n\rightarrow \infty} \Vert W_n \Vert_{\op} = 2$ almost surely.

\end{proof}

\subsection{Proof of Theorem \ref{thm:SBM1}}\label{sec:proof_sbm}

\begin{proof} The proof is similar to the proof of Theorem \ref{thm:Z2sync1}, where the GOE matrix $W_n$ is replaced by the noise matrix $E$. 

Applying Theorem \ref{thm:apprxconcavepoint} with the matrix $\oA_G(\lambda)$, similar to Equation (\ref{eq:nonvanishing1}), we have
\begin{equation}
\begin{aligned}
\frac{\lambda}{n} \Vert \sigma^\sT u \Vert_2^2 \geq \left(1-\frac{1}{k-1}\right) \SDP\left(\oA_G(\lambda) \right) - \frac{1}{k-1} \SDP\left( -E\right) - \SDP\left(E\right).
\end{aligned}
\end{equation}

According to \cite[Theorem 8]{montanari2016semidefinite}, the gap between the SDPs with the two different noise matrices is bounded with high probability by a function of the average degree $d$
\[
\left\vert \frac{1}{n} \SDP ( \oA_G (\lambda) ) - \frac{1}{n} \SDP ( A(\lambda)) \right\vert < C\frac{\log d}{d^{1/10}} \quad \text{ and } \quad \left\vert \frac{1}{n} \SDP (\pm E) - \frac{1}{n} \SDP \left(\pm W_n\right) \right\vert < C\frac{\log d}{d^{1/10}},
\]
where $A(\lambda) = \lambda/n \cdot u u^\sT + W_n$ corresponds to the $\mZt$ synchronization model and $C = C(\lambda)$ is a function of $\lambda$ bounded for any fixed $\lambda$. 

According to \cite[Theorem 5]{montanari2016semidefinite}, for any $\delta > 0$ and $\lambda > 1$, there exists a function $\Delta (\lambda) >0$ such that with high probability, we have
\begin{align}
 \frac{1}{n}\SDP \left( \pm W_n \right)  \leq 2 + \delta, \quad \text{and} \quad 
 \frac{1}{n}\SDP \left( \frac{\lambda}{n} u u^\sT + W_n \right)  \geq 2 + \Delta(\lambda). 
\end{align}

Combining the above results, we have for any $\delta > 0$, with high probability
\[
\inf_{\sigma \in \Crit_{n,k}} \frac{1}{n^2} \Vert \sigma^\sT u \Vert_2^2 \geq \left(1-\frac{1}{k-1} \right) \frac{\Delta(\lambda)}{\lambda}  - \frac{4 + \delta}{k-1} \cdot \frac{1}{\lambda} - \frac{\delta}{\lambda} - 2 \frac{C(\lambda)}{\lambda} \cdot \frac{\log d}{d^{1/10}}.
\]
For a sufficiently small $\eps > 0$, taking $\delta$ sufficiently small, and taking successively $d$ and $k$ sufficiently large, the above expression will be greater than $\eps$, which concludes the proof.
\end{proof}

\subsection{Proof of Theorem \ref{thm:SOd_grothendieck}}\label{sec:proof_sod}

We decompose the proof into three parts. In the first part, we do the calculation for a general non-convex problem. In the second part, we focus on the non-convex problem (\ref{Ncvx-OC-SDP}). In the third part, we prove a claim we made in the second part.

\begin{proof}

\subsubsection*{Part 1}

First, let's consider a general SDP problem. Given a symmetric matrix $A \in \R^{n\times n}$, symmetric matrices $B_1, B_2, \ldots, B_s \in \R^{n\times n}$ and real numbers $c_1, \ldots, c_s \in \R$, we consider the following SDP:
\begin{equation}\label{SDPGen}
\begin{aligned}
\max_{X \in \R^{n \times n}} & \quad \langle A, X \rangle\\
\text{subject to} & \quad  \langle B_i, X \rangle = c_i, \quad i \in [s], \\
& \quad X \succeq 0.
\end{aligned}
\end{equation}
Let $\vB = [B_1,\ldots, B_s] $ and $\vc = (c_1,\ldots, c_s)$. We denote $\SDP(A, \vB, \vc)$ the maximum of the above SDP problem:
\[
\SDP(A, \vB, \vc) = \max \{ \langle A, X \rangle: X \succeq 0, \langle B_i, X \rangle = c_i, i \in [s] \}. 
\]
We assume $\SDP(A, \vB, \vc) < \infty$. 

For a fixed integer $k$, the Burer-Monteiro approach considers the following non-convex problem:
\begin{equation}\label{OPTkGen}
\begin{aligned}
\text{maximize} & \quad \func( \sigma)=\langle \sigma, A \sigma \rangle\\
\text{subject to} & \quad  \langle \sigma, B_i \sigma \rangle = c_i, \quad i \in [s],
\end{aligned}
\end{equation}
with decision variable $\sigma \in \R^{n\times k}$

Define the manifold $\cM_k^{\vB,\vc} = \{ \sigma \in \R^{n\times k}: \langle \sigma, B_i \sigma \rangle = c_i, i \in [s] \}$. At each point $\sigma \in \cM_k^{\vB,\vc}$, the tangent space is given by $ T_\sigma \cM_k^{\vB,\vc} = \lbrace U \in \R^{n\times k}: \langle U , B_i \sigma \rangle = 0, i \in [s] \rbrace$. We denote $\proj_T (U)$ the projection of $U \in \R^{n \times k}$ onto $T_\sigma \mathcal{M}_k^{\vB,\vc}$:
\[
\proj_T (U) = U - \sum_{i,j = 1}^s M_{ij} \langle B_j \sigma, U \rangle B_i \sigma.
\]
where $M = ((\langle B_i\sigma, B_j \sigma \rangle)_{ij=1}^s)^{-1} \in \R^{s \times s}$. The Riemannian gradient is therefore given by
\[
\tgrad \func (\sigma) = 2 (A - \sum_{i=1}^s \lambda_i B_i ) \sigma
\]
with $\lambda_i = \sum_{j=1}^s M_{ij} \tr(B_j A \sigma\sigma^\sT)$. We will write $\Lambda = \Lambda (\sigma) = \sum_{i=1}^s \lambda_i B_i$. The Riemannian Hessian $\thess \func(\sigma)$ applied on the direction $U \in T_\sigma \cM_k^{\vB,\vc}$ gives
\[
\langle U, \thess \func(\sigma) [U]\rangle = 2 \langle U, (A - \Lambda) U\rangle.
\]
Therefore, according to the definition of the $\eps$-approximate concave point $\sigma \in \mathcal{M}_k^{\vB,\vc}$, we have
\[
\forall Y \in T_\sigma \mathcal{M}_k^{\vB,\vc}, \quad \langle Y, (\Lambda - A) Y \rangle \geq -\frac{1}{2}\eps \langle Y, Y \rangle.
\]

Let $V \in \mathbb R^{n\times n}$ such that $X^* = V V^\sT$ is a solution of the general SDP problem (\ref{SDPGen}), and $G \in \mathbb R^{k\times n}$, $G_{ij} \sim \normal(0, 1/k)$ i.i.d., be a random mapping from $\R^n$ onto $\R^k$. Let $\sigma$ be a local maximizer of the rank-$k$ non-convex SDP (\ref{OPTkGen}), and take $Y = \proj_T (V G^\sT)$, a random projection of $V \in R^{n \times n}$ onto $T_{\sigma}\mathcal{M}_k^{\vB,\vc}$. Due to the definition of the approximate concave point, we have
\[
\E \langle \proj_T (VG^\sT), (\Lambda - A) \proj_T(VG^\sT)\rangle \geq -\frac{\eps}{2} \E \langle \proj_T (VG^\sT), \proj_T (VG^\sT)\rangle 
\]
where the expectation is taken over the random mapping $G$. Expanding the left hand side gives
\begin{equation}\label{eq:coreequation}
\begin{aligned}
0 \leq &\langle V, (\Lambda - A) V \rangle - 2 \E \langle VG^\sT - \proj_T(VG^\sT), (\Lambda - A) VG^\sT \rangle \\
&+ \E \langle VG^\sT - \proj_T(VG^\sT), (\Lambda - A) (VG^\sT - \proj_T(VG^\sT)) \rangle + \frac{\eps}{2} \E \langle \proj_T (VG^\sT), \proj_T (VG^\sT)\rangle .
\end{aligned}
\end{equation}
%Note that for any feasible point $Y \in \mathcal{M}_k^{\vB,\vc}$, we have $\func( \sigma) = \langle \sigma, A \sigma \rangle = \langle Y, \Lambda Y\rangle$. Therefore, the first term of the above equation gives
%\[
%\begin{aligned}
%\langle V, (\Lambda - A) V \rangle
%= &\func( \sigma) - \SDP(A, \vB, \vc) \\
%\end{aligned}
%\]

The second term in the last equation gives
\begin{equation}\label{eq:2rdterm}
\begin{aligned}
&\E \left\langle VG^\sT - \proj_T \left(VG^\sT\right), \left(\Lambda - A\right) VG^\sT \right\rangle \\
=& \E \left\langle \sum_{i,j=1}^s M_{ij}  \left\langle B_j \sigma, VG^\sT \right\rangle B_i \sigma, \left(\Lambda - A\right) VG^\sT\right\rangle\\
=& \sum_{i,j=1}^s M_{ij} \E \left[\left\langle B_i \sigma , \left(\Lambda - A\right) VG^\sT\right\rangle \left\langle B_j \sigma, VG^\sT \right\rangle\right]\\
%=& \sum_{i,j=1}^s M_{ij} \E\left[ \tr\left( \sigma^\sT B_i \left(\Lambda - A\right) VG^\sT\right) \tr\left(\sigma^\sT B_j VG^\sT\right)\right]\\
=& \sum_{i,j=1}^s M_{ij} \E\left[ \left\langle V^\sT \left(\Lambda - A\right) B_i \sigma, G^\sT \right\rangle \left\langle V^\sT B_j \sigma, G^\sT \right\rangle\right]\\
=& \frac{1}{k} \sum_{i,j = 1}^s M_{ij} \left\langle V^\sT \left(\Lambda - A\right) B_i \sigma, V^\sT B_j \sigma \right\rangle \\
=& \frac{1}{k}  \left\langle \left(\Lambda - A\right), V V^\sT  \sum_{i,j = 1}^s M_{ij} B_j \sigma \sigma^\sT B_i \right\rangle. \\
\end{aligned}
\end{equation}

The third term gives
\begin{equation}\label{eq:3rdterm}
\begin{aligned}
&\E \left\langle VG^\sT - \proj_T\left(VG^\sT\right), \left(\Lambda - A\right) \left(VG^\sT - \proj_T\left(VG^\sT\right)\right) \right\rangle \\
=& \E \left\langle \sum_{i,j=1}^s M_{ij} \left\langle B_j \sigma, VG^\sT \right\rangle  B_i \sigma, \left(\Lambda - A\right)  \sum_{k,l=1}^s M_{k l}  \left\langle B_l \sigma, VG^\sT \right\rangle B_k \sigma \right\rangle\\
=& \sum_{ijkl = 1}^s M_{ij} M_{kl} \E \left[\left\langle B_i \sigma , \left(\Lambda - A\right) B_k \sigma \right\rangle \left\langle B_j \sigma, VG^\sT \right\rangle \left\langle B_l \sigma, V G^\sT \right\rangle\right]\\
=&  \frac{1}{k}\sum_{ijkl = 1}^s M_{ij} M_{kl} \left\langle B_i , \left(\Lambda - A\right) B_k \sigma\sigma^\sT \right\rangle \left\langle X^* B_j , B_l \sigma \sigma^\sT\right\rangle.
\end{aligned}
\end{equation}

For the fourth term, we have
\begin{equation}\label{eq:4thterm}
\begin{aligned}
&\E \langle \proj_T (VG^\sT), \proj_T (VG^\sT)\rangle \\
=& \E \Big\Vert  VG -  \sum_{i,j=1}^s M_{ij} \left\langle B_j \sigma, VG^\sT \right\rangle  B_i \sigma \Big\Vert_F^2\\
=& \E \left \langle VG, VG\right \rangle - \E \Big \Vert  \sum_{i,j=1}^s M_{ij} \left\langle B_j \sigma, VG^\sT \right\rangle  B_i \sigma \Big \Vert_F^2\\ 
=& n - \sum_{ij=1}^n \sum_{kl=1}^n M_{ij} M_{kl} \left\langle B_i \sigma, B_k \sigma \right\rangle \E \left[ \left\langle V^\sT B_j \sigma, G \right\rangle \left\langle V^\sT B_l \sigma, G \right\rangle \right]\\
=& n - \frac{1}{k} \sum_{ij=1}^n \sum_{kl=1}^n M_{ij} M_{kl} \left \langle B_i \sigma, B_k \sigma \right\rangle \cdot \left\langle V^\sT B_j \sigma, V^\sT B_l \sigma \right\rangle.
\end{aligned}
\end{equation}

\subsubsection*{Part 2}

Now let's consider the case of the rank-$k$ non-convex Orthogonal-Cut SDP problem (\ref{Ncvx-OC-SDP}). There are $s = d(d+1)/2\times m$ constraints corresponding to the set $\{ (B_i,c_i): i \in [s]\} = \{ (E_{ii}, 1): i \in [n]\} \bigcup \cup_{t=1}^m\{ ((E_{ij} + E_{ji})/\sqrt{2},0): (t-1) d+1 \le i < j \le t d\}$, where $E_{ij} = e_i e_j^\sT$. We will denote $\cM_{o,d,k}$ the optimization manifold:
\[
\cM_{o,d,k} = \lbrace \sigma \in \R^{md \times k} : \sigma = ( \sigma_1 , \ldots , \sigma_m )^\sT , \sigma_i^\sT \sigma_i = \id_d, i\in [m] \rbrace. 
\]
It is straightforward to verify that for any $\sigma \in \cM_{o,d,k}$, we have $(\langle B_i \sigma, B_j \sigma \rangle)_{ij=1}^s = \id_s$. Thus, we have $M = \id_s$. In the following calculation, we write $X = \sigma \sigma^\sT$. Recall that $X^*$ is a global maximizer of problem (\ref{OC-SDP}), and $X^* = V V^\sT$. 

Now, let us calculate each term in Equation (\ref{eq:coreequation}), for the specific problem (\ref{Ncvx-OC-SDP}). For the second term in Equation (\ref{eq:coreequation}), we derived Equation (\ref{eq:2rdterm}). One can check with some calculations that for any $\sigma \in \cM_{o,d,k}$, we have
\[
\sum_{i,j=1}^s M_{ij} B_j \sigma\sigma^\sT B_i = \frac{d+1}{2} \id_n.
\] 
For the fourth term in Equation (\ref{eq:coreequation}), we derived Equation (\ref{eq:4thterm}). Following the calculation in Equation (\ref{eq:4thterm}), we have
\[
\begin{aligned}
&\sum_{ij=1}^n \sum_{kl=1}^n M_{ij} M_{kl} \left \langle B_i \sigma, B_k \sigma \right\rangle \cdot \left\langle V^\sT B_j \sigma, V^\sT B_l \sigma \right\rangle\\
=& \sum_{i j=1}^n  \left \langle B_i \sigma, B_j \sigma \right\rangle\cdot \left\langle V^\sT B_i \sigma, V^\sT B_j \sigma \right\rangle  = \sum_{i = 1}^n   \left\langle V^\sT B_i \sigma, V^\sT B_i \sigma \right\rangle =  \frac{d+1}{2}n. 
\end{aligned}
\]
For the third term in Equation (\ref{eq:coreequation}), we derived Equation (\ref{eq:3rdterm}). Following the calculation in Equation (\ref{eq:3rdterm}), we have
\[
\begin{aligned}
&\sum_{ijkl=1}^s M_{ij}M_{kl}\langle B_i, (\Lambda-A) B_k X \rangle \langle X^* B_j , B_l X \rangle\\
=& \sum_{kl = 1}^s \langle B_k, (\Lambda-A) B_l X \rangle \langle X^* B_k , B_l X \rangle\\
=& \sum_{ij=1}^n (\Lambda-A)_{ij} \sum_{kl=1}^s \langle B_k, E_{ij} B_l X \rangle \langle X^* B_k, B_l X \rangle\\
=& \frac{d+1}{2}\tr\left((\Lambda-A)\barX\right)
\end{aligned}
\]
where we define $\barX = 2/(d+1) \cdot (\sum_{kl=1}^s \langle B_k, E_{ij} B_l X \rangle \langle X^* B_k, B_l X \rangle)_{i,j=1}^n$. Here, we claim that $\barX$ is a feasible point of the Orthogonal-Cut SDP problem (\ref{OC-SDP}). We will prove this claim in part $3$. 

For any feasible point $X$ of the Orthogonal-Cut SDP problem (\ref{OC-SDP}), we have $ \func( \sigma) = \langle \Lambda, X \rangle$. Therefore, from Equation (\ref{eq:coreequation}), we obtain
\[
\begin{aligned}
&f(\sigma) - \SDP_o(A)  = \langle V \Lambda V \rangle - \langle V, A V \rangle \\
\geq & \frac{1}{k}\left( 2 \cdot \frac{d+1}{2}\langle (\Lambda - A), V V^\sT \rangle - \frac{d+1}{2} \langle (\Lambda-A), \barX \rangle \right) - \frac{n}{2}\eps \left(1 - \frac{d+1}{2k}\right)\\
= & \frac{1}{k} \left((d+1)(f(\sigma) - \SDP_o(A)) - \frac{d+1}{2} (f(\sigma) -\langle A, \barX ) \right)- \frac{n}{2}\eps \left(1 - \frac{d+1}{2k}\right)\\
\geq &  \left( 2 \frac{d+1}{2k}(f(\sigma) - \SDP_o(A)) - \frac{d+1}{2k} (f(\sigma) +\SDP_o(-A) ) \right) - \frac{n}{2}\eps \left(1 - \frac{d+1}{2k}\right)
\end{aligned}
\]
Letting $k_d = 2k/(d+1)$, rearranging the above inequality, we have
\[
\left(1-\frac{1}{k_d}\right) \func( \sigma) - \left(1-\frac{2}{k_d}\right)\SDP_o(A) + \frac{1}{k_d}  \SDP_o(-A) \geq - \frac{n}{2}\eps \left(1 - \frac{1}{k_d}\right),
\]
which finally gives the desired inequality
\[
\func( \sigma) \geq \SDP_o(A) - \frac{1}{k_d-1} (\SDP_o(A) + \SDP_o(-A)) - \frac{n}{2}\eps. 
\]

\subsubsection*{Part 3}

Now, let us check that $\barX$ is a feasible point of the Orthogonal-Cut SDP problem (\ref{OC-SDP}). The reason is given by the following Fact $(a)$ and $(b)$. 

\noindent
{\bf Fact $(a)$}. $\barX$ is P.S.D.. Indeed, for any $v \in \R^n$, recall that $X = \sigma \sigma^\sT$ and $X^* = V V^\sT$, we have 
\[
\begin{aligned}
\langle v, \barX v \rangle =& \frac{2}{d+1} \cdot \sum_{kl=1}^s \langle v^\sT B_k \sigma, v^\sT B_l \sigma \rangle \langle V^\sT B_k \sigma, V^\sT B_l \sigma \rangle \\
=& \frac{2}{d+1} \cdot \tr((\langle v^\sT B_k \sigma, v^\sT B_l \sigma \rangle)_{k,l=1}^s  \cdot (\langle V^\sT B_k \sigma, V^\sT B_l \sigma \rangle)_{k,l=1}^s).
\end{aligned}
\]
The matrix $(\langle v^\sT B_k \sigma, v^\sT B_l \sigma \rangle)_{k,l=1}^s = Z^\sT Z \succeq 0$, where $Z = [\text{vec}(v^\sT B_1 \sigma), \ldots, \text{vec}(v^\sT B_s \sigma)]$. Similarly, we have $(\langle V^\sT B_k \sigma, V^\sT B_l \sigma \rangle)_{k,l=1}^s \succeq 0$. Thus, $\langle v, \barX v\rangle \geq 0 $ for any $v \in \R^n$. Then $\barX$ is P.S.D..

\noindent
{\bf  Fact $(b)$} The $(i,i)$'th block of $\barX$ equals $\id_d$. To show this, we assume $d\geq 2$, and due to the symmetry, we just need to check $\barX_{11} = 1$ and $\barX_{12} = 0$. We denote $J_{ij} = E_{ii} \delta_{ij} + (E_{ij} + E_{ji})/\sqrt 2 \cdot (1-\delta_{ij})$, and we rewrite $\barX_{ij}$ as
\[
\barX_{ij} = \frac{2}{d+1} \cdot \sum_{a = 1}^m \sum_{(k,s,l,t) \in \Gamma_a}\langle E_{ij},  J_{ks} X J_{lt} \rangle \langle X^*,  J_{ks} X  J_{lt} \rangle
\]
where $\Gamma_a = \{(k,s,l,t): 1 + (a-1)d\leq k\leq s \leq ad, 1 + (a-1)d \leq l \leq t \leq ad \}$. We have the following series of simplification
\[
\begin{aligned}
\barX_{11} = & \frac{2}{d+1}  \cdot \sum_{a = 1}^m \sum_{(k,s,l,t) \in \Gamma_a}\langle E_{11},  J_{ks} X J_{lt} \rangle \langle X^*,  J_{ks} X  J_{lt} \rangle\\
=& \frac{2}{d+1}  \cdot  \sum_{(k,s,l,t) \in \Gamma_1} \langle E_{11},  J_{ks} X J_{lt} \rangle \langle X^*,  J_{ks} X  J_{lt} \rangle \\
=& \frac{2}{d+1}  \cdot  \sum_{(k,s,l,t) \in \Gamma_1} \langle E_{11},  J_{ks}J_{lt} \rangle \langle J_{ks}, J_{lt} \rangle =  \frac{2}{d+1}  \cdot  \sum_{1 \leq k \leq s \leq d} \langle E_{11},  J_{ks} J_{ks} \rangle \\
=& \frac{2}{d+1}  \cdot  \left(\sum_{k=1}^d \langle E_{11},  E_{kk} E_{kk} \rangle + \sum_{1\leq k < l \leq d} \frac{1}{2} \langle E_{11}, E_{kk} + E_{ll} \rangle \right)\\
=& \frac{2}{d+1}  \cdot \left(1 + \frac{d-1}{2}\right) = 1.
\end{aligned}
\]
The third equality used the fact that $X$ and $X^*$ are feasible point so that their $(i,i)$'th block are $\id_d$. Similarly, we have
\[
\begin{aligned}
\barX_{12} = & \frac{2}{d+1}  \cdot  \sum_{(k,s,l,t) \in \Gamma_1} \langle E_{12},  J_{ks} X J_{lt} \rangle \langle X^*,  J_{ks} X  J_{lt} \rangle \\
=& \frac{2}{d+1}  \cdot  \sum_{(k,s,l,t) \in \Gamma_1} \langle E_{12},  J_{ks}J_{lt} \rangle \langle J_{ks}, J_{lt} \rangle 
= \frac{2}{d+1}  \cdot  \sum_{1 \leq k \leq s \leq d} \langle E_{12},  J_{ks} J_{ks} \rangle = 0.
\end{aligned}
\]
The last equality is because $J_{ks} J_{ks}$ is always a diagonal matrix. 

Therefore, we proved that $\barX$ is a feasible point of the Orthogonal-Cut SDP problem (\ref{OC-SDP}). 
\end{proof}

\subsection{Proof of Theorem \ref{thm:trustregion}}

Given a point $\sigma \in \mathcal{M}_k$ and a tangent vector $u \in T_{\sigma} \mathcal{M}_k$ with $\Vert u \Vert_F = 1$, we denote $\sigma (t) = \proj_{\mathcal{M}_k} (\sigma + t u)$ the update  with searching direction $u$ and step size $t$. The next three lemmas ensure a sufficient increment of the objective function at each step of the RTR algorithm.

%The details of the bounds on the derivatives of $\func( \sigma(t))$ are deferred to Appendix \ref{app:derivativesbounds}. 

%We will use the following general upper bounds on the Hessian:
%\[
%\lambda_H(\sigma) \leq \Vert \thess (\func( \sigma)) \Vert_2 \leq \Vert A - \Lambda \Vert_2 \leq 2 \Vert A \Vert_1
%\]
%In particular, $\lambda_H$ will always be upper bounded by $2 \Vert A \Vert_1$. 

\begin{lemma} (Gradient-step) Fix $\mu_G \leq 2 \Vert A \Vert_1$. For any point $\sigma \in \mathcal{M}_k$ such that $\Vert \tgrad  \func( \sigma) \Vert_F \geq \mu_G$, taking searching direction $u = \tgrad  \func( \sigma)  / \Vert \tgrad  \func( \sigma) \Vert_F$ and step size $\eta =  \mu_G/(20 \Vert A \Vert_1)$, we have
\[
\func( \sigma(\eta)) - \func( \sigma) \geq  \frac{\mu_G^2}{40 \Vert A \Vert_1}. 
\]
\label{lem:grad}
\end{lemma}
\begin{proof}
The second order expansion of $\func( \sigma(t))$ around $0$ with $t \leq 1$ gives 
\[
\begin{aligned}
\quad \func( \sigma(t)) - \func( \sigma) & \geq (\func\circ  \sigma)'(0) t - \sup_{\xi \in [0,t]}\frac{1}{2}(\func\circ  \sigma)''(\xi) t^2 \\
& \geq \langle \tgrad \func( \sigma), u \rangle t - \frac{1}{2}\Vert A \Vert_1 \cdot (4  + 8 t + 8t^2 )\cdot t^2 \\
& \geq  \Vert \tgrad \func( \sigma) \Vert_F t - 10 \Vert A \Vert_1 t^2.
\end{aligned}
\]
The second inequality used the bound on the second order derivative in Lemma \ref{lem:secondderivative} in Appendix \ref{app:derivativesbounds}. Now we take $t  = \mu_G /(20 \Vert A \Vert_1)$. Since $\mu_G \leq 2 \Vert A \Vert_1$, we have $t \leq 1$. Plugging this $t$ into the above equation completes the proof.
\end{proof} 

\begin{lemma} (Eigen-step) For any point $\sigma \in \mathcal{M}_k$, and $u\in T_{\sigma} \mathcal{M}_k$ satisfying $\Vert u \Vert_F = 1$, $\langle u , \tgrad \func( \sigma)\rangle \geq 0$, and $\lambda_H = \lambda_H(\sigma, u) = \thess \func( \sigma)[u,u] >0$, choosing $\eta =  \lambda_H/(100 \Vert A \Vert_1)$, we have
\[
\func( \sigma (\eta)) - \func( \sigma) \geq \frac{\lambda_H^3}{4 \cdot 10^4 \Vert A \Vert_1^2}. 
\]
\label{lem:hess1}
\end{lemma}
\begin{proof}
The third order expansion of $\func( \sigma(t))$ around $0$ for $t \leq 1$ gives 
\[
\begin{aligned}
\func( \sigma(t)) - \func( \sigma) & \geq \langle \tgrad \func( \sigma(0)),u \rangle t + \frac{1}{2}\langle u, \thess \func( \sigma(0)) [u]\rangle t^2 - \frac{1}{6}  \sup_{\xi \in [0,t]} (\func\circ  \sigma)'''(\xi) t^3 \\
& \geq \frac{1}{2} \lambda_H t^2 -  \frac{1}{6} \Vert A \Vert_1
 \cdot (12 + 36t + 48 t^2 + 48 t^3)\cdot t^3\\
& \geq  \frac{1}{2} \lambda_H t^2 - 24 \Vert A \Vert_1 t^3. 
\end{aligned}
\]
The second inequality used the bound on the third order derivative in Lemma \ref{lem:thirdderivative} in Appendix \ref{app:derivativesbounds}. Now we take $t  = \lambda_H/(100 \Vert A \Vert_1)$. Note that we always have $\lambda_H(\sigma, u) \leq \Vert \thess \func( \sigma) \Vert_2 \leq \Vert A - \Lambda \Vert_2 \leq 2 \Vert A \Vert_1$, and therefore we have $t \leq 2 \Vert A \Vert_1 /(100 \Vert A \Vert_1) \leq 1$. Plugging this $t$ into the above equation completes the proof.
\end{proof} 

The last lower bound on the increment of objective function for eigen-step used the loose bound in Lemma \ref{lem:thirdderivative}. Using Lemma \ref{lem:improvedthird}, we can give an improved bound for the eigen-step when the norm of the gradient is small. In particular we take $\mu_G = \Vert A \Vert_2$.

\begin{lemma} (Improved bound for eigen-step)
For any point $\sigma \in \mathcal{M}_k$ with $\Vert \tgrad  \func( \sigma) \Vert_F \leq \mu_G = \Vert A \Vert_2$, and $u\in T_{\sigma} \mathcal{M}_k$ satisfying $\Vert u \Vert_F = 1$, $\langle u , \tgrad \func( \sigma)\rangle \geq 0$ and $\lambda_H = \lambda_H(\sigma, u) = \thess f(\sigma)[u,u]  >0$, choosing $\eta =   \min \left(\sqrt{\lambda_H/(216 \Vert A \Vert_1)},\lambda_H/(12 \Vert A \Vert_2)\right)$, we have
\[
\func( \sigma (\eta)) - \func( \sigma) \geq  \frac{1}{4}\lambda_H \eta^2 = \min \left(\frac{\lambda_H^2}{864 \Vert A \Vert_1}, \frac{\lambda_H^3}{576 \Vert A \Vert_2^2} \right). 
\]
\label{lem:hess2}
\end{lemma}

\begin{proof}
The third order expansion of $\func( \sigma(t))$ around $0$ for $t \leq 1$ gives 
\[
\begin{aligned}
\func( \sigma(t)) - \func( \sigma) &\geq \langle \tgrad \func( \sigma(0)),u \rangle t + \frac{1}{2}\langle u, \thess \func( \sigma(0)) [u]\rangle t^2 - \frac{1}{6}  \sup_{\xi \in [0,t]} (\func\circ  \sigma)'''(\xi) t^3 \\
& \geq \frac{1}{2} \lambda_H t^2 -  \frac{1}{6} (6\Vert A \Vert_2 + 3 \Vert \tgrad \func( \sigma(0))\Vert_F ) \cdot t^3 - \frac{1}{6} \Vert A \Vert_1 \cdot (42  + 72 t + 48 t^2) \cdot t^4 \\
& \geq  \frac{1}{2} \lambda_H t^2 - \frac{3}{2} \Vert A \Vert_2 t^3 - 27 \Vert A \Vert_1 t^4.
\end{aligned}
\]
The first inequality used the improved bound on the third order derivative of Lemma \ref{lem:improvedthird} in Appendix \ref{app:derivativesbounds},
which imply in particular $\|\tgrad f(\sigma)\|_F\le \|A\|_2$. 
Taking $t  = \min \left(\sqrt{\lambda_H/(216 \Vert A \Vert_1)},\lambda_H/(12 \Vert A \Vert_2)\right)<1$ completes the proof.
\end{proof}

We are now at a good position to prove Theorem \ref{thm:trustregion}.

\begin{proof} 
Denote $\func^* = \SDP (A) -1/(k-1) \cdot (\SDP(A) + \SDP(-A))$ and $g(\sigma ) = \func^* - \func( \sigma)$. Let $T$ be the number of iterations and $\lbrace\sigma^0 , \sigma^1 ,\ldots , \sigma^T\rbrace \subset \mathcal{M}_k$ the iterates returned by our RTR algorithm from an arbitrary initialization $\sigma^0 \in \mathcal{M}_k$. We are only interested in the convergence rate as $g(\sigma) > 0$, namely the convergence rate below the gap. Since our algorithm is an ascent algorithm, without loss of generality, we assume $g(\sigma^0), \ldots, g(\sigma^T) > 0$ (otherwise the theorem will hold automatically). 

At each point $\sigma \in \mathcal{M}_k$, Theorem \ref{thm:apprxconcavepoint} gives the following lower bound on the highest curvature
\[
\lambda_{H,\max}(\sigma) = \sup_{u\in T_\sigma \mathcal{M}_k} \frac{\langle u , \thess  \func( \sigma)  [u]\rangle}{\langle u,u \rangle} \geq 2\frac{g(\sigma)}{n} > 0.
\]
We will use this information to bound the algorithm's convergence rate. 

%\noindent
{\bf Case 1. }
First, we consider the case when all the RTR steps are eigen-steps. In each iteration, the algorithm constructs an update direction $u^t$ with curvature $\lambda_H(\sigma^t, u^t) \geq \lambda_{H,\max}(\sigma^t)/2$. According to Lemma \ref{lem:hess2}, we have
\[
g(\sigma^{t}) - g(\sigma^{t+1}) \geq \frac{\lambda_H^3(\sigma^t)}{32\cdot 10^4 \Vert A \Vert^2_1} \geq \frac{g(\sigma^t)^3}{32 \cdot 10^4 \Vert A \Vert^2_1 n^3},
\]
which implies $g(\sigma^{t+1}) \leq g(\sigma^t)$. Thus, we have 
\[
\frac{1}{g(\sigma^{t+1})^2} - \frac{1}{g(\sigma^t)^2} \geq \left(\frac{g(\sigma^{t})^2}{g(\sigma^{t+1})^2} + \frac{g(\sigma^{t})}{g(\sigma^{t+1})} \right)\cdot  \frac{1}{32 \cdot 10^4 \Vert A \Vert_1^2 n} 
 \geq \frac{1}{16 \cdot 10^4 \Vert A \Vert_1^2 n^3}. 
\]
Summing over $t = 0,\ldots, T-1$, we have
\[
\frac{1}{g(\sigma^T)^2} - \frac{1}{g(\sigma^0)^2} = \sum_{0 \leq t \leq T-1} \frac{1}{g(\sigma^{t+1})^2} - \frac{1}{g(\sigma^t)^2} \geq \frac{1}{16 \cdot 10^4 \Vert A \Vert_1^2 n^3} T
\]
Therefore, we obtain the convergence rate $g(\sigma^T) \leq 400 \Vert A \Vert_1 n \sqrt{n/T}$. This implies that 
\[
\func( \sigma^T) \geq \SDP(A) - \frac{1}{k-1} (\SDP(A) + \SDP(-A)) - \frac{n}{2} \eps
\]
as soon as $T \geq 64 \cdot 10^4 \cdot n \Vert A \Vert_1^2/\eps^2$.

%\noindent
{\bf Case 2. }
Then, we consider the case where we set $\mu_G = \Vert A \Vert_2$, and we use the gradient step as $\Vert \tgrad \func( \sigma) \Vert_F > \mu_G$, and use the eigen-step as $\Vert \tgrad \func( \sigma) \Vert_F \leq \mu_G$. First let us bound the number of gradient steps. According to Lemma \ref{lem:grad}, we have 
\[
T_G \frac{\mu_G^2}{40 \Vert A \Vert_1} \leq g(\sigma^0) - g(\sigma^T) \leq \RGD (A).
\]
Hence, we deduce the upper bound $T_G \leq 40 \cdot \Vert A\Vert_1 \RGD(A)/\Vert A \Vert_2^2$. 

Then let us bound the number of eigen-steps. Let us denote $\mathcal I$ and $\mathcal  J \subset \lbrace 0, 1, \ldots, T-1 \rbrace$ the subsets of indices corresponding to eigensteps with respectively $\lambda_H \geq 3\Vert A \Vert_2^2/(2\Vert A \Vert_1)$ and $\lambda_H < 3\Vert A \Vert_2^2/(2\Vert A \Vert_1)$. According to Lemma \ref{lem:hess2}, we have for all $t\in \mathcal  J$
\[
\frac{1}{g(\sigma^{t+1})} - \frac{1}{g(\sigma^t)} \geq \frac{1}{864 \Vert A \Vert_1 n^2} \frac{g(\sigma^{t})}{g(\sigma^{t+1})} \geq \frac{1}{864 \Vert A \Vert_1 n^2},
\]
whereas for $t \in \mathcal I$
\[
\frac{1}{g(\sigma^{t+1})^2} - \frac{1}{g(\sigma^t)^2} \geq \frac{1}{576 \Vert A \Vert_2^2 n^3} \left(\frac{g(\sigma^{t})}{g(\sigma^{t+1})} + \frac{g(\sigma^{t})^2}{g(\sigma^{t+1})^2}\right) \geq \frac{1}{288 \Vert A \Vert_2^2 n^3}.
\]
Summing the contributions of the above two equations gives the convergence rate
\[
g(\sigma^T) \leq c \cdot \max \left(\Vert A \Vert_1 \frac{n^2}{T}, \Vert A \Vert_2 n \sqrt{\frac{n}{T}}\right)
\]
for a universal constant $c$. This guarantees that $g(\sigma^T) \leq n\eps/2$ as soon as $T \geq \tilde c \cdot n \max \left( \Vert A \Vert_2^2/\eps^2, \Vert A \Vert_1/\eps \right)$ for some universal constant $\tilde c$. 

\end{proof}

\section*{Acknowledgements}

A.M. was partially supported by the NSF grant CCF-1319979. S.M. was supported by Office of Technology Licensing Stanford Graduate Fellowship. 

\bibliographystyle{amsalpha}
\newcommand{\etalchar}[1]{$^{#1}$}
\providecommand{\bysame}{\leavevmode\hbox to3em{\hrulefill}\thinspace}
\providecommand{\MR}{\relax\ifhmode\unskip\space\fi MR }
\providecommand{\MRhref}[2]{%
  \href{http://www.ams.org/mathscinet-getitem?mr=#1}{#2}
}
\providecommand{\href}[2]{#2}

\appendix

\section{Some technical steps}
\label{sec:bounds}

\subsection{Technical lemmas on $(\func\circ  \sigma)(t) $}
\label{app:derivativesbounds}

In this section, we give an upper bound to the second and third derivatives of $(\func\circ \sigma) (t) = \langle \sigma(t), A \sigma(t) \rangle$ (these notations are defined below). These bounds are important in bounding the complexity of the Riemannian trust-region method in solving the non-convex SDP problem. 

Fix a point $\sigma \in \mathcal{M}_k \subset \R^{n\times k}$ on the manifold, and a tangent vector $u  \in T_\sigma \mathcal{M}_k = \{u = [u_1, \ldots, u_i]^\sT \in \R^{n \times k}: u_i \in \R^k, \langle \sigma_i, u_i \rangle = 0, \forall i \in [n] \}$ with $\Vert u \Vert_F = 1$. Let $\sigma(t) = \proj_{\mathcal{M}_k}(\sigma + t u)$ be the orthogonal projection of $\sigma + t u$ onto the manifold $\cM_k$. For a given symmetric matrix $A \in \R^{n \times n}$, let $\func( \sigma) = \langle \sigma, A \sigma\rangle$. We would like to study the derivatives of $(\func\circ  \sigma) (t) = \func( \sigma(t))$ with respect to $t$. Furthermore, we define $u_i(t) = u_i/\sqrt{1 + t^2 \Vert u_i \Vert_2^2}$, $u (t) = [ u_1 (t) ,\ldots , u_n (t)]^\sT$, $D(t) = \diag([\Vert u_1(t) \Vert_2^2,\ldots, \Vert u_n(t) \Vert_2^2 ])$, and $\Lambda(t) = \ddiag(A\sigma(t) \sigma(t)^\sT)$. For convenience, we will denote $\tsigma = \sigma(t)$, $\tu = u(t)$, $\tD = D(t)$, and $\tLambda = \Lambda(t)$. 

\begin{lemma}
For any $\sigma \in \mathcal{M}_k$ and $u \in T_\sigma \mathcal{M}_k$, let $\sigma(t) = \proj_{\mathcal{M}_k}(\sigma + t u)$. We have $\forall t \in \R$
\begin{equation}
\begin{aligned}
\sigma'(t) =& - tD(t) \sigma(t) + u(t), \\
\sigma''(t) =& \left[-D(t) + 3 t^2 D(t)^2 \right] \sigma(t) - 2 t D(t) u(t), \\
\sigma'''(t) = & \left[9 t D(t)^2 - 15 t^3 D(t)^3\right] \sigma(t) + \left[-3D(t) + 9 t^2 D(t)^2\right] u(t). \\
\end{aligned}
\end{equation}
\label{lem:derivatives}
\end{lemma}

\begin{proof}
To calculate the first three derivatives of $\sigma (t)$, we expand each row of $\sigma(t+r)$ up to third order in $r$:
\[
\begin{aligned}
\sigma_i(t+r) = & \frac{\sigma_i + t u_i + r u_i}{\sqrt{1 + (t+r)^2 \Vert u_i \Vert_2^2 }} \\
=& \frac{\sigma_i + tu_i + ru_i}{\sqrt{1 + t^2 \Vert u_i \Vert_2^2} } \cdot \left( 1- \frac{1}{2} \cdot \frac{2rt \Vert u_i \Vert_2^2 + r^2 \Vert u_i \Vert_2^2}{1 + t^2 \Vert u_i \Vert_2^2} +\frac{3}{8} \left(\frac{2rt \Vert u_i \Vert_2^2 + r^2 \Vert u_i \Vert_2^2}{1 + t^2 \Vert u_i \Vert_2^2}\right)^2 \right. \\
&\left. - \frac{5}{16} \left(\frac{2rt \Vert u_i \Vert_2^2 + r^2 \Vert u_i \Vert_2^2}{1 + t^2 \Vert u_i \Vert_2^2}\right)^3 + o(r^3)\right)\\
= &  \sigma_i(t) + \left\lbrace \left[- t \Vert u_i(t) \Vert_2^2\right] \sigma_i(t) + u_i(t) \right\rbrace r \\
& + \left\lbrace \left[-\frac{1}{2}\Vert u_i(t) \Vert_2^2 + \frac{3}{2} t^2 \Vert u_i(t)\Vert_2^4\right] \sigma_i(t) + \left[- t \Vert u_i(t)\Vert_2^2\right] u_i(t) \right\rbrace r^2 \\
& + \left\lbrace \left[\frac{3}{2} t \Vert u_i(t)\Vert_2^4 - \frac{5}{2} t^3 \Vert u_i(t) \Vert_2^6\right] \sigma_i(t) + \left[-\frac{1}{2}\Vert u_i(t) \Vert_2^2 + \frac{3}{2} t^2 \Vert u_i(t)\Vert_2^4\right] u_i(t) \right\rbrace r^3 + o(r^3).
\end{aligned}
\]
By matching each expansion coefficient to the corresponding derivative, we obtain the desired result.
\end{proof}

\begin{lemma}%For any $\sigma \in \mathcal{M}_k$ and $u \in T_\sigma \mathcal{M}_k$ with $\Vert u \Vert_F = 1$. 
For $(\func\circ  \sigma)(t)$ as defined above
\begin{align}
\sup_{\xi \in [0,t]} \vert (\func\circ  \sigma )'' (\xi) \vert \leq \Vert A \Vert_1 \cdot  (4 + 8 t  + 8t^2), \quad \forall t \geq 0. 
\end{align}
\label{lem:secondderivative}
\end{lemma}

\begin{proof} We explicitly calculate the second derivative
\[
\begin{aligned}
(\func\circ  \sigma)''(t) =& \langle \sigma'(t), \nabla^2 \func( \sigma(t)) [ \sigma'(t)] \rangle + \langle \nabla \func( \sigma(t)), \sigma''(t) \rangle\\ 
=& \langle \sigma'(t), 2 A  \sigma'(t) \rangle + \langle 2 A \sigma(t), \sigma''(t) \rangle\\
=& \langle - t \tD \tsigma + \tu, 2 A  [- t\tD \tsigma + \tu] \rangle + \langle 2 A  \tsigma, [-\tD + 3 t^2 \tD^2 ] \tsigma - 2 t \tD \tu \rangle\\
=& 2 \langle \tu, (A - \tLambda) \tu \rangle - 4t [\langle \tu, A \tD \tsigma \rangle + \langle A \tsigma, \tD \tu \rangle] + t^2 [2 \langle \tD\tsigma, A\tD \tsigma \rangle + 6 \langle A \tsigma, \tD^2 \tsigma \rangle ]. 
\end{aligned}
\]
Noticing that $\Vert u (t) \Vert_F \leq \Vert u (0) \Vert_F = 1$, we can use the bounds derived in Appendix \ref{app:bounds} to obtain the following inequality
\[
\begin{aligned}
\left\vert (\func\circ  \sigma )''(t) \right\vert &\leq 2 |\langle \tu, (A - \tLambda) \tu \rangle | + 4t [ |\langle \tu, A \tD \tsigma \rangle| + |\langle A \tsigma, \tD \tu \rangle | ] + t^2 [2 |\langle \tD\tsigma, A\tD \tsigma \rangle | + 6 |\langle A \tsigma, \tD^2 \tsigma \rangle | ]\\
&\leq 4 \Vert A \Vert_1 + 8 t \Vert A \Vert_1 + 8t^2 \Vert A \Vert_1.
\end{aligned}
\]
\end{proof}

\begin{lemma}
For $(\func\circ  \sigma)(t)$ as defined above
\begin{align}
\sup_{\xi \in [0,t]} \vert (\func\circ  \sigma )''' (\xi) \vert \leq \Vert A \Vert_1 \cdot (12 + 36 t + 48 t^2 + 48 t^3),\quad \forall t \geq 0. 
\end{align}
\label{lem:thirdderivative}
\end{lemma}

\begin{proof} We explicitly calculate the third derivative
\[
\begin{aligned}
(\func\circ  \sigma)'''(t) 
=&\langle \nabla \func( \sigma(t)), \sigma'''(t) \rangle + 3 \langle \sigma'(t), \nabla^2 \func( \sigma(t))[\sigma''(t)] \rangle\\
=&\langle \nabla \func( \tsigma), [9 t \tD^2 - 15 t^3 \tD^3] \tsigma + [-3\tD + 9 t^2 \tD^2 ] \tu \rangle \\
&+ 3 \langle- t\tD \tsigma + \tu, \nabla^2 \func( \tsigma) [(-\tD + 3t^2 \tD^2)\tsigma - 2t\tD\tu]  \rangle\\
=& -6 [\langle \tsigma, A \tD \tu \rangle + \langle \tu, A\tD \tsigma \rangle]+ [18 \langle \tsigma, A \tD^2 \tsigma \rangle + 6 \langle \tD \tsigma, A \tD \tsigma \rangle - 12 \langle \tu, A\tD \tu \rangle] t\\
&+ [18 \langle \tsigma, A\tD^2 \tu \rangle + 12 \langle \tD\tsigma, A\tD \tu \rangle + 18 \langle \tu, A \tD^2 \tsigma \rangle]t^2 \\
&+ [-30 \langle \tsigma, A \tD^3 \tsigma \rangle - 18 \langle \tD \tsigma ,A \tD^2 \tsigma \rangle]t^3.
\end{aligned}
\]
The inequality is obtained by upper bounding each term using the bounds derived in Appendix \ref{app:bounds}.
\end{proof}

The above bound on the third derivative of order $\Vert A \Vert_1$ as $t\to 0$. The next lemma proves a bound of order $\Vert A \Vert_2 + \Vert \tgrad \func( \sigma(0))\Vert_F$ as $t\to 0$. If $\Vert \tgrad \func( \sigma(0))\Vert_F$ is small, this improves the above bound. 

\begin{lemma}
For $(\func\circ  \sigma)(t)$ as defined above, an improved bound on its third derivative gives  
\begin{align}
\sup_{\xi \in [0,t]} \vert (\func\circ  \sigma )''' (\xi) \vert \leq 6 \Vert A \Vert_2 + 3 \Vert \tgrad \func( \sigma(0))\Vert_F + \Vert A \Vert_1 \cdot (42 t  + 72 t^2 + 48 t^3), \quad \forall t \geq 0. 
\end{align}
\label{lem:improvedthird}
\end{lemma}
\begin{proof}
From the proof in the previous lemma, we have
\[
\vert (\func\circ  \sigma)''' (t) \vert \leq 6 \left( |\langle \tsigma, A \tD \tu \rangle| + |\langle \tu, A\tD \tsigma \rangle | \right)  +\Vert A \Vert_1 \cdot (36 t + 48 t^2 + 48 t^3)\, .. 
\]

We next bound more carefully $g(t) = \langle \sigma(t), A D(t) u(t) \rangle$. Simple calculation gives us
\[
u'(t) = - t \cdot D(t) u(t) \quad \text{and} \quad D'(t) = -2t \cdot D(t)^2.
\]
Hence,
\[
\begin{aligned}
g'(t) =& \langle \sigma'(t), AD(t) u(t) \rangle + \langle \sigma(t), A D'(t) u(t) \rangle + \langle \sigma(t), A D(t) u'(t)\rangle \\
=& \langle -t \tD \tsigma + \tu, A\tD\tu \rangle + \langle \sigma, A(-2t\tD^2) \tu\rangle + \langle \tsigma, A \tD (-t\tD \tu) \rangle\\
=& \langle \tu, A\tD\tu \rangle + [-\langle \tD\tsigma, A\tD\tu\rangle -3\langle \tsigma,A\tD^2 \tu \rangle] t.
\end{aligned}
\]
According to the bounds in Appendix \ref{app:bounds}, we have
\[
\vert g'(t) \vert \leq \Vert A \Vert_2 + 4 t \Vert A \Vert_1. 
\]
In the meanwhile, we have
\[
\vert g(0)\vert = \vert \langle \sigma, A D u \rangle \vert =  \vert \langle (A - \Lambda)\sigma, D u \rangle  \vert = \vert \langle \frac{1}{2}\tgrad \func( \sigma(0)), Du \rangle \vert \leq  \frac{1}{2} \Vert \tgrad \func( \sigma(0))\Vert_F. 
\]
According to the Taylor expansion of $g(t)$ around $0$ and $t\geq 0$ at first order, we have
\[
\vert g(t) \vert \leq \vert g(0) \vert + t \sup_{\xi \in [0,t]} \vert g'(\xi) \vert \leq \frac{1}{2} \Vert \tgrad \func( \sigma(0))\Vert_F + t\Vert A \Vert_2 + 4 t^2 \Vert A \Vert_1.
\]
Hence the improved bound follows.
\end{proof}

\subsection{All the bounds}
\label{app:bounds}

In this section we give all the bounds used in the proof of Lemma \ref{lem:secondderivative}, \ref{lem:thirdderivative} and \ref{lem:improvedthird}. Let $\sigma \in \cM_k$ and $u \in \R^{n \times k}$ with $\Vert u \Vert_F \leq 1$. Note that here we do not require $u \in T_\sigma \cM_k$. Denote $D = \diag([\Vert u_1\Vert_2^2, \ldots, \Vert u_n \Vert_2^2])$ and $\Lambda = \ddiag (A \sigma \sigma^\sT )$. We have the following bound for each term. 

\begin{multicols}{2}
\begin{enumerate}
\item  
\[
\begin{aligned}
 \vert \langle \sigma, AD u \rangle \vert =& \vert \langle Du, A \sigma \rangle \vert \\
\leq & \max_i \Vert (A \sigma)_i \Vert_2\\
\leq & \Vert A \Vert_1. 
\end{aligned}
\]
\item 
\[
\begin{aligned}
\vert \langle u, AD \sigma \rangle \vert
=& \vert \langle u, \ddiag(A u \sigma^{\sT}) u \rangle \vert\\
\leq& \max (\vert \diag(Au\sigma^{\sT})\vert)\\
=& \max_i \vert \langle \sigma_i, (Au)_i\rangle \vert \\
\leq& \max_i \Vert (Au)_i \Vert_2\\
\leq& \Vert Au \Vert_F \\
\leq& \Vert A\Vert_2 \Vert u \Vert_F \\
\leq & \Vert A \Vert_2. 
\end{aligned}
\]
\item
\[
\begin{aligned}
\vert \langle u, AD u \rangle\vert \leq & \Vert ADu \Vert_F \\
\leq& \Vert A \Vert_2 \Vert Du\Vert_F \\
\leq & \Vert A \Vert_2. 
\end{aligned}
\]
\item
\[
\begin{aligned}
\vert \langle D\sigma, AD \sigma \rangle \vert 
= & \vert \langle u, \ddiag(A D \sigma \sigma^\sT) u \rangle\vert \\
\leq & \max_i(\vert\ddiag(A D \sigma \sigma^\sT)_{ii}\vert) \\
\leq &\max_i \Vert (AD\sigma)_i\Vert_2\\ 
\leq & \Vert AD \Vert_1\\
\leq & \vert A \vert_\infty\\
\leq& \Vert A \Vert_1.
\end{aligned}
\]

\item
\[
\begin{aligned}
\vert \langle A \sigma, D^2 \sigma\rangle\vert
= & \vert \langle u, \ddiag(A \sigma \sigma^{\sT}) D u\rangle \vert \\
\leq & \max_i (\vert \ddiag(A \sigma \sigma^{\sT})_{ii} D_{ii}\vert) \\
\leq & \max_i(\ddiag(A \sigma \sigma^{\sT})_{ii})\\
\leq & \max_i (\langle \sigma_i, (A \sigma)_i \rangle )\\
\leq & \max_i \Vert (A\sigma)_i \Vert_2\\
\leq & \Vert A \Vert_1.
\end{aligned}
\]

\item
\[
\begin{aligned}
\vert \langle u, (A - \Lambda) u \rangle\vert 
\leq & \Vert A - \Lambda \Vert_2 \\
\leq & \Vert A \Vert_2 + \max_{i} \vert \ddiag(A \sigma\sigma^{\sT})_{ii}\vert \\
\leq& \Vert A \Vert_2 + \Vert A \Vert_1\\
\leq & 2 \Vert A \Vert_1. 
\end{aligned}
\]

\item
\[
\begin{aligned}
\vert \langle D \sigma, AD u \rangle \vert = & \vert \langle AD \sigma, Du\rangle\vert \\
\leq & \max_i \Vert (AD\sigma)_i \Vert_2 \\
\leq & \Vert A \Vert_1. 
\end{aligned}
\]

\item
\[
\begin{aligned}
\vert \langle \sigma, AD^2 u \rangle \rangle \vert
\leq & \max_i \Vert (A \sigma)_i \Vert_2 \\
\leq & \Vert A \Vert_1. 
\end{aligned}
\]

\item
\[
\begin{aligned}
\vert \langle u, AD^2 \sigma \rangle\vert
\leq & \max_i \Vert (AD^2 \sigma)_i \Vert_2 \\
\leq & \Vert A D^2 \Vert_1 \\
\leq& \Vert A \Vert_1. 
\end{aligned}
\]

\item
\[
\begin{aligned}
\vert \langle \sigma, AD^3 \sigma \rangle\vert
= & \vert \langle u, \ddiag(A \sigma \sigma^{\sT}) D^2 u \rangle \vert \\
\leq & \max(\ddiag(A \sigma \sigma^{\sT}) D^2) \\
\leq & \max(\ddiag(A \sigma \sigma^{\sT})) \\
\leq & \Vert A \Vert_1. 
\end{aligned}
\]

\item
\[
\begin{aligned}
\vert \langle D \sigma, AD^2 \sigma \rangle\vert
= & \vert \langle u, \ddiag(A D \sigma \sigma^{\sT}) D u\rangle \vert \\
\leq& \max_i (\vert \ddiag(A D \sigma \sigma^{\sT})_{ii} D_{ii}\vert) \\
\leq& \max_i \vert \langle  \sigma_i, (AD \sigma)_i \rangle\vert \\
\leq & \Vert A D \Vert_1 \\
\leq & \Vert A \Vert_1. 
\end{aligned}
\]
\end{enumerate}
\end{multicols}

\end{document}